\documentclass{article}

\usepackage[a4paper]{geometry}
\usepackage[utf8]{inputenc}

\usepackage{graphicx}
\usepackage{latexsym}

\usepackage{lipsum}
\usepackage{graphicx}
\usepackage{authblk}
\usepackage[english]{babel}
\usepackage{color}
\usepackage{hyperref}
\hypersetup{
    colorlinks=true,
    linkcolor=blue,
    filecolor=blue,      
    urlcolor=blue,
    citecolor=blue,
    }
\usepackage{dsfont}
\usepackage{geometry}
\usepackage{float}
\usepackage{bbm}

\usepackage[nottoc]{tocbibind}
 \usepackage{indentfirst}
\usepackage{customcommands}
\usepackage{amsmath}

\graphicspath{{./figure/}}

\usepackage{pgfplots}
\usepackage{comment}

\usepackage{tikz}
\usepackage{caption}
\usepackage{subcaption}
\DeclareCaptionLabelFormat{custom}
{%
      \textbf{Figure #2}
}
\DeclareCaptionLabelSeparator{custom}{ }

\DeclareCaptionFormat{custom}
{%
\centering
    #1 #2 #3 
}
\usepackage{amsmath,amsfonts,amssymb,amsthm}

\numberwithin{equation}{section}
\usepackage{etoolbox}
\apptocmd{\thebibliography}{}{}{}

\newcommand{\R}{{\mathbb R}}

\theoremstyle{plain}
\newtheorem{theorem}{Theorem}[section]

\newtheorem{corollary}[theorem]{Corollary}
\newtheorem{proposition}[theorem]{Proposition}

\newtheorem{reminder*}{[theorem]Reminder}

\newtheorem{details*}[theorem]{Details}

\newtheorem{comm*}{Comment}
 
\newtheorem{definition*}{[theorem]Definition}

\newtheorem{notation*}{Notation}

\newtheorem{remark}[theorem]{Remark}

\newcommand{\vertii}[1]{{\left\vert\kern-0.3ex\left\vert #1 
    \right\vert\kern-0.3ex\right\vert}}

\newcommand{\vertiii}[1]{{\left\vert\kern-0.3ex\left\vert\kern-0.3ex\left\vert #1 
    \right\vert\kern-0.3ex\right\vert\kern-0.3ex\right\vert}}

\newcommand*\circled[1]{\tikz[baseline=(char.base)]{\node[shape=circle,draw,inner sep=2pt] (char) {#1};}}

\title{One-dimensional inelastic collapse of four particles:\\asymmetric collision sequences and spherical billiard reduction}
\author{Th\'eophile Dolmaire, Eleni H\"ubner-Rosenau}
\date{\today}

\begin{document}

\maketitle

\begin{abstract}
\noindent
We consider a one-dimensional system of four inelastic hard spheres, colliding with a fixed restitution coefficient $r$, and we study the inelastic collapse phenomenon for such a particle system. We study a periodic, asymmetric collision pattern, proving that it can be realized, despite its instability. We prove that we can associate to the four-particle dynamical system another dynamical system of smaller dimension, acting on $\{1,2,3\} \times \mathbb{S}^2$, and that encodes the collision orders of each trajectory. We provide different representations of this new dynamical system, and study numerically its $\omega$-limit sets. In particular, the numerical simulations suggest that the orbits of such a system might be quasi-periodic.
\end{abstract}


\section{Introduction}

\subsection{The model and the main results}
\noindent
In the present article we consider a system of four identical inelastic hard spheres, evolving on the one-dimensional line $\mathbb{R}$, and such that the collision law is defined with a fixed restitution coefficient $r \in\ ]0,1[$. In other words, when a collision occurs, the velocities $v,v_* \in\mathbb{R}$ of two colliding particles are immediately changed into $v',v_*' \in\mathbb{R}$ according to the relations:
\begin{align}
\label{EQUAT_Law_Collision_}
\left\{
\begin{array}{ccc}
v' + v'_* &=& v+v_*,\\
v'-v'_* &=& -r(v-v_*).
\end{array}
\right.
\end{align}
We study the phenomenon of inelastic collapse for such a system, which corresponds to a situation where an infinite number of collisions take place in finite time. Such a phenomenon, still not well understood although of central importance, is characteristic of inelastic particle systems, in particular when the restitution coefficient is fixed.\\
More specifically, we will study some particular inelastic collapses, taking place according to a prescribed order of collisions, involving specific pairs of particles. We will focus on collision orders that are periodic, but not symmetric, in the sense that the two external particles of the system are not involved in the same number of collisions during a period, contrary to what was done so far in the literature.\\
Besides, in the same spirit of \cite{CoGM995}, we will study the billiard representation of the particle system. We will identify a hidden conserved quantity of the system, which will allow us to reduce the dimension of the system. Thanks to this reduction, we prove that the order of the collisions of any trajectory can be recovered from the study of a billiard on a portion of the unit sphere, which can also be seen as a dynamical system defined on $\{1,2,3\} \times \mathbb{S}^2$. We call this dynamical system the spherical reduction mapping. Numerical investigations on such a system enable to recover the existing results of \cite{CDKK999}, and allow in addition to observe the $\omega$-limit sets of the spherical reduction mapping. One of the most intriguing observation is the apparent quasi-periodic behaviour of the orbits, along possible invariant tori.

\subsection{Motivation and derivation of the collision model}
\noindent
The inelastic hard sphere system with constant restitution coefficient is one of the simplest microscopic models to describe granular media. Therefore, understanding the properties of such a particle system is of utmost importance to reach a better understanding of the kinetic theory of granular materials. In turn, granular materials constitute fundamental models in Physics (the reader may find in \cite{CHMR021} a recent mathematical survey on the theory of granular gases): their applications stretch from the description of snow or large collections of grains of sand (see \cite{Kada999}), to clouds of interstellar dust or even the rings of Saturn (see \cite{GrBr984}). In addition, granular materials exhibit fascinating properties, and their behaviour shares some features with solids, liquids and gases at the same time (see \cite{BeJN996}). One typical behaviour of granular media is the spontaneous onset of inhomogeneities (\cite{BrPo004}), that can already be observed at the level of the particle system (\cite{GoZa993}, \cite{McYo996}).\\
\newline
However, let us emphasize that the model describing collisions taking place immediately, with a fixed restitution coefficient, is an idealization of real collisions, that take place more precisely between the time of the first contact between the particles, and the later time when they finally separate. In between, complex mechanisms of deformation act on the particles, during which the kinetic energy of the particles is first converted into potential energy or temperature. Then, the compression of the particles is converted back into a generally different amount of kinetic energy, causing the separation of the particles. A description of this mechanism in terms of forces of compression was proposed by Hertz in \cite{Hert882}, providing a conservative model, adapted later in a dissipative version in \cite{BSHP996}.\\
Nevertheless, the time scale of the collisions is much smaller than the time scale of the evolution of the whole particle system. Therefore, only the relations that give the post-collisional velocities in terms of the pre-collisional velocities are relevant to model the dynamics of the particles. The normal restitution coefficient can then be defined as the ratio of the normal components of the post- and pre-collisional relative velocities between the colliding particles. Such a restitution coefficient is determined from the force of compression for example in \cite{ScPo998}, \cite{RPBS999} or \cite{BrPo004}. Nevertheless, in general, due to friction, a collision modifies not only the normal component of the relative velocity, but also the tangential components. In the same way, one can define and determine the tangential restitution coefficient (\cite{BSPo008}, \cite{SBPo008}), and in general both normal and tangential restitution coefficients are functions of the norm $\vert v-v_* \vert$ of the relative velocity.\\
In any case, in the one-dimensional case the notion of tangential velocity does not make sense, and in the present article, we will not consider such advanced models, with variable (normal) restitution coefficient. Actually, it is well-known that the collision model with a fixed restitution coefficient is already a good approximation of real granular materials (see for instance \cite{BrPo004}).

\subsection{Existing results about inelastic collapse}
\noindent
The question of the well-posedness of the dynamics of the particle system is still an open problem (contrary to the elastic case $r=1$, for which Alexander proved the global in time well-posedness in any dimension for almost every initial configuration of the particles, see \cite{Alex975} and \cite{Alex976}). The dynamics of hard spheres is constructed iteratively: from a configuration, one computes the next collision time, pushes forwards the particles until such a collision, then updates the velocities of the colliding particles, and reiterates the process. In the case of inelastic hard spheres with fixed restitution coefficients, the phenomenon of \emph{inelastic collapse} can occur, when infinitely many collisions take place in finite time. This phenomenon is a major obstruction to obtain a well-posedness result: when inelastic collapse happens, defining the dynamics beyond such a singularity yields an ill-posed problem.\\
\newline
The inelastic collapse was first observed and mathematically described in \cite{ShKa989} and \cite{BeMa990}, in one-dimensional models. Later, the first observation of inelastic collapse in dimension larger than $1$ was obtained numerically in \cite{McYo993}\footnote{Let us mention also \cite{TrBa000}, in which the authors observe numerically inelastic collapse in dimensions $4$, $5$ and $6$. Such an investigation originated in a previous conjecture claiming that inelastic collapse cannot happen in dimension larger than $3$.}. In the same article, the authors identified also that the particles involved in the collapse are distributed along a linear structure. Therefore, and although it can be observed in dimension $2$, the mechanism of inelastic collapse seems intrinsically one-dimensional. For this reason, the study of inelastic collapse in dimension $1$ has attracted a lot of attention, and the present article constitutes a continuation of this study.\\
Let us review briefly the literature concerning the one-dimensional inelastic collapse. Except when it will be mentioned explicitly, all the works that follow studied the collision law \eqref{EQUAT_Law_Collision_}, with fixed restitution coefficient.\\ The case of three particles was discussed in \cite{McYo991}, and \cite{CoGM995}, after which the system can be considered as completely understood. In particular, it is shown that if $r > 7 - 4\sqrt{3}$, no collapse can take place. Below this value, the collapse takes place for a set of initial configurations of positive measure.\\
In the case of four particles, as noticed in \cite{McYo991}, the two-particle system with boundary studied in \cite{BeMa990} can be seen as a symmetric system of four particles (hence, associated to set of zero measure of initial data), and one can show that such a system cannot collapse if $r > 3-2\sqrt{2}$. One of the classical approaches to study inelastic collapse is to determine a priori an order of collisions between the particles, which is periodic, so that the evolution of the velocities of the system is given by a linear iteration. In the case of four particles, if we denote respectively by $\mathfrak{a}$, $\mathfrak{b}$ and $\mathfrak{c}$ the collisions between the pairs \circled{1}-\circled{2}, \circled{2}-\circled{3} and \circled{3}-\circled{4}, contrary to the case of three particles, the order of collisions is not predetermined, and many sequences composed of $\mathfrak{a}$, $\mathfrak{b}$ or $\mathfrak{c}$ are a priori admissible candidates. In \cite{CDKK999}, the approach of Cipra, Dini, Kennedy and Kolan started with numerical simulations to determine if any collision sequence was appearing, which would suggest that there exists a set of initial data of positive measure generating such a sequence. The authors observed indeed some periodic sequences, and (surprisingly) all of them were of the form $(\mathfrak{ab})^n(\mathfrak{cb})^n$, with $n \geq 2$. Then, by mathematical analysis (see also \cite{HuRo023} for an exhaustive discussion of the method), Cipra et al. proved that the periodic sequences $(\mathfrak{ab})^n(\mathfrak{cb})^n$ were indeed feasible and stable, that is, realizable from a set of initial configurations of positive measure (at least for $n$ small). Besides, they observed no other periodic and stable sequence. In addition, the authors discovered that each sequence $(\mathfrak{ab})^n(\mathfrak{cb})^n$ is stable only if $r$ belongs to some particular interval $I_n$, these intervals being disjoint, and converging seemingly towards the critical value $r = 7-4\sqrt{3} \simeq 0.07180$ of the collapsing three-particle system. The upper bound of $I_2$ was found to be $3-2\sqrt{2} \simeq 0.17157$, and the authors did not find any other stable pattern for $r$ larger than this value. Nevertheless, in \cite{CDKK999}, it is proven that there exist initial configurations leading to a collapse, with the collision sequence $\mathfrak{abcb}$, as long as $r < r_\text{exist.}$, with $r_\text{exist.} \simeq 0.19166$ being a root of an explicit polynomial of degree $6$. It might be the critical restitution coefficient for four-particle systems, although there is no proof of such a result. To the best of our knowledge, no other example of stable pattern is known in the literature for four-particle systems.\\
Let us also mention \cite{BeCa999}, in which the existence of a stable collapse (that is, the existence of a set of initial configurations of positive measure leading to the inelastic collapse of four particles) is proven with an approach different from \cite{CDKK999}.\\
Concerning larger systems of particles, only special, and very symmetric constructions of collapsing systems are known. It is not clear if a perturbation of an initial configuration leading to such a symmetric collapse still leads to collapse. Such constructions can be found in \cite{BeCa999}, and also in the recent \cite{ChKZ022}, based on an elegant geometrical construction. In particular, there is no example of a stable collapse for systems composed of five particles or more.\\
Concerning estimates about the critical restitution coefficient $r_\text{crit.}(n)$, that is, the largest restitution coefficient for which a collapse can occur in a system of $n$ particles, the first discussions in this direction can be found in \cite{BeMa990} and \cite{McYo991}, lower and upper bounds are proved in \cite{BeCa999}, and the construction of \cite{ChKZ022} enables also to deduce estimates. However, for $n \geq 4$ the question of the exact value remains open.\\
Let us finally conclude the discussion concerning the one-dimensional particle systems with constant restitution coefficient by mentioning \cite{GrMu996}, in which the authors study a system of three particles evolving on a circle, discovering an extremely rich dynamics.\\
\newline
As for the case of higher dimensions, until recently, the only existing reference about inelastic particle systems was \cite{ZhKa996}, in which the authors derive geometrical necessary conditions on the final positions of the particles for the collapse to occur. Recently, \cite{DoVeAr1} and \cite{DoVeAr2} completed the study of the three-particle system in any dimension, proving in particular in \cite{DoVeAr1} the existence of a set of initial configurations of positive measure, leading to inelastic collapse in dimension $d\geq 2$.\\
\newline
Let us finally mention results concerning analogous models. In \cite{ScZh996}, the authors consider three rotating inelastic particles in dimension $d \geq 2$, and the collision law involves a tangential restitution coefficient in addition to the normal restitution coefficient (generalizing therefore the model of \cite{ZhKa996} and \cite{DoVeAr1}). The authors found that inelastic collapse can occur, even for $r$ close to $1$.\\
As for models with a restitution coefficient depending on $\vert v-v_* \vert$, it is known that accurate models should satisfy $r \rightarrow 1$ as $\vert v-v_* \vert \rightarrow 0$. It has been observed that for such models, inelastic collapse seems impossible (see \cite{BrPo004}, \cite{PoBS003}). Such a result is proved rigorously in \cite{GSBM998} for a system of three one-dimensional particles, evolving either on the real line or on a circle. An analogous result is proved in \cite{DoVeNot}, in which inelastic collisions take place only if the relative velocity is large enough, and such that the loss of kinetic energy per inelastic collision is a fixed quantity.

\subsection{Outline of the article}
\noindent
The plan of the present article is the following. In the second section, we introduce the particle system we will consider. In the third section, using the method introduced in \cite{CDKK999} we study in detail the asymmetric pattern $\mathfrak{ababcb}$. We prove the first main result of this article: such a pattern can be achieved for some initial configurations and appropriate restitution coefficients $r$, but its self-similar realization is never stable. In the fourth section, we represent the trajectories of the four-particle system as a billiard in the first octant in $\mathbb{R}^3$. We prove that all the initial configurations $\left(p(0),q(0)\right)$, (where $p(0)$ denotes the vector of the initial positions, and $q(0)$ the vector of the initial velocities) that span the same plane $\mathcal{P}(0) = \text{Span}\left(p(0),q(0)\right)$ generate trajectories that, at each collision $k$, span the same plane $\mathcal{P}(k)$. This is the second main result of the article, which allows to associate to the four-particle system another dynamical system of smaller dimension, the \emph{spherical reduction mapping}, that encodes all the possible collision orders. In the fifth section, we perform numerical simulation with the spherical reduction mapping, studying in particular the $\omega$-limit sets of the trajectories. We observe a behaviour suggesting that the orbits are quasi-periodic.

\section{Presentation of the model}
\noindent
We consider a system of four inelastic hard spheres evolving on the one-dimensional line, all having the same mass (that we will assume to be equal to $1$ without loss of generality), and all having zero radius (the particles are described by points in $\mathbb{R}$). We will denote the respective positions of the particles  \circled{1} up to \circled{4} at time $t$ by $x_1(t),\dots,x_4(t) \in \mathbb{R}$ and their respective velocities by $v_1(t),\dots,v_4(t) \in \mathbb{R}$. Considering the initial positions, we will label the particles from left to right so that we have in particular:
\begin{align*}
x_1(0) \leq x_2(0) \leq x_3(0) \leq x_4(0).
\end{align*}
Starting from an initial configuration of the system at time $t=0$, with positions $\left(x_1(0),x_2(0),x_3(0),x_4(0)\right)$ and velocities $\left(v_1(0),v_2(0),v_3(0),v_4(0)\right)$, let us describe the dynamics of the particle system. As long as all the particles are separated, that is, when $\vert x_{i}(t) - x_j(t) \vert > 0$ $\forall i\neq j \in \{1,2,3,4\}$, we assume that the particles evolve according to the free dynamics, that is, the particles have all an inertial motion, moving at constant velocity. Therefore, as long as the separation of the particles holds, we have:
\begin{align}
	v_i(t)&=v_i(0) \quad && \forall i \in \{1,2,3,4\} \label{EQUATFree_TrnspVariable_v}\\ 
	x_i(t)&=x_i(0)+t v_i(0) \quad && \forall i \in \{1,2,3,4\} \label{EQUATFree_TrnspVariable_x}.
\end{align}
When the separation does not hold, that is, when there exists a time $t_1 \geq 0$ such that two particles collide, we have: $x_{i+1}(t_1) - x_i(t_1) = 0$ for (at least) one certain index $i \in \{1,2,3\}$. In such a case, we will assume that the respective velocities of the particles \circled{$i$} and \circled{$j$}, with $j=i+1$, are immediately changed, according to the following laws:
\begin{align}
v_i(t_1^+) + v_{i+1}(t_1^+) &= v_i(t_1^-) + v_{i+1}(t_1^-) \quad && \text{(conservation of momentum)},\label{EQUAT_Law_ColliCnsrvMomnt}\\
v_i(t_1^+) - v_{i+1}(t_1^+) &= -r\left( v_i(t_1^-) - v_{i+1}(t_1^-)\right) \quad && \text{(loss of a fixed fraction of the relative velocity)},\label{EQUAT_Law_ColliNormalComp}\\
v_k(t_1^+) &= v_k(t_1^-) \quad && \forall k \neq i,i+1.
\end{align}
where $r$ is a fixed real number, with $r \in\ ]0,1[$, and $v_i(t_1^-)$ and $v_i(t_1^+)$ denoting the velocity of the particle \circled{$i$}, respectively, right before and right after the collision.\\
The velocities $v_i(t_1^-)$ and $v_{i+1}(t_1^-)$ are called the \emph{pre-collisional velocities}, while the velocities $v_i(t_1^+)$ and $v_{i+1}(t_1^+)$ are called the \emph{post-collisional velocities}.\\
After such a collision, the particles are immediately separated once again, and we will assume that they will evolve again according to the free dynamics, until the time $t_2 > t_1$ of the next collision, and so on.\\
The fixed parameter $r$ is called the \emph{restitution coefficient}, and such a particle system is called the \emph{inelastic hard sphere system, with fixed restitution coefficient}.\\
\newline
Let us observe that the positions $x_i$ of the particles are time-continuous functions, while the velocities are piecewise constant functions, possibly discontinuous at the collision times $t_k$.\\
By convention, we define the velocities at the collision times $t_k$ to be equal to the post-collisional velocities, that is, we fix:
\begin{align}
\label{EQUATConventionPostC_Velo}
v_i(t_k) = v_i(t_k^+).
\end{align}

\begin{remark}
According to our assumptions, only adjacent particles can collide, and in order to have a first collision at time $t_1$ between the particles \circled{$i$} and \circled{$j$} with $j = i+1$ initially separated with $x_j(0) > x_i(0)$, we necessarily have $v_j(0) - v_i(0) < 0$. \eqref{EQUAT_Law_ColliNormalComp} implies that right after the collision time $t_1$ we have $v_j(t_1^+) - v_i(t_1^+) > 0$, so that the particles separate again, and we have $x_j(0) > x_i(0)$ for $t > t_1$ small enough.\\
Therefore, the particles cannot cross each other, and the order of the particles on the real line holds for all times. In other words, we have for all times:
\begin{align}
x_1(t) \leq x_2(t) \leq x_3(t) \leq x_4(t).
\end{align}
\end{remark}

\begin{remark}
Let us observe that the collision law \eqref{EQUAT_Law_ColliCnsrvMomnt}, \eqref{EQUAT_Law_ColliNormalComp} yields a well-posed dynamics only if each colliding particle collides with a single other particle. Such a collision is called a \emph{binary collision}. When a particle collides with two or more other particles, such a collision is called a \emph{multiple collision}.\\
However, this does not mean that only two particles are involved in a collision at a given time: in the present case, the pairs of particles \circled{1}-\circled{2} and \circled{3}-\circled{4} can perfectly collide at the same time.
\end{remark}

\begin{remark}
We supposed that the system is composed of point particles. Nevertheless, since the particles evolve on the real line, we could have considered particles \circled{$i$} with positive radii $r_i$, possibly different one from another. We would have then recovered the same dynamics, provided that the mass of all the particles remains the same. The only difference would have been in the definition of the relative positions. A collision at time $t$ would correspond instead to the condition $x_{i+1}(t) - x_i(t) - (r_i+r_{i+1}) = 0$.
\end{remark}

\noindent
Following, for instance, the approach of \cite{CoGM995} or \cite{CDKK999}, instead of considering the respective positions $x_i(t)$ and velocities $v_i(t)$ of the four particles, we will describe the system with the help of the relative distances and velocities, for adjacent particles. Namely, we introduce:
\begin{align}
p_i(t) &= x_{i+1}(t) - x_i(t),\hspace{2mm} \forall\, 1 \leq i \leq 3,\\
q_i(t) &= v_{i+1}(t) - v_i(t),\hspace{2mm} \forall\, 1 \leq i \leq 3,
\end{align}
and we define the \emph{position vector} $p(t) \in (\mathbb{R}_+)^3$ and the \emph{velocity vector} $q(t) \in \mathbb{R}^3$ as:
\begin{align}
p(t) = \left(p_1(t),p_2(t),p_3(t)\right) \hspace{3mm} \text{and} \hspace{3mm} q(t) = \left(q_1(t),q_2(t),q_3(t)\right).
\end{align}
Between two collision times $t_k$ and $t_{k+1}$ the evolution laws \eqref{EQUATFree_TrnspVariable_v}-\eqref{EQUATFree_TrnspVariable_x} write:
\begin{align}
	q(t)&=q(t_k) && \text{ for }t_k\leq t<t_{k+1}, \label{EQUATFree_TrnspVariable_q}\\ 
	p(t)&=p(t_k)+(t-t_k)q(t_k) \quad && \text{ for }t_k\leq t\leq t_{k+1} \label{EQUATFree_TrnspVariable_p}.
\end{align}
and when a collision takes place (that is, $p_i(t_{k+1})=0$ for some $i \in \{1,2,3\}$), the collision law \eqref{EQUAT_Law_ColliCnsrvMomnt}-\eqref{EQUAT_Law_ColliNormalComp} can be rewritten in the following matricial form:
\begin{align}\label{eq: update q vector}
	q(t_{k+1}^+)=K q(t_{k+1}^-),
\end{align}
where the matrix $K\in \R^{3\times 3}$ is equal to one of the three following \emph{collision matrices}:
\begin{align*}
	\begin{split}
		A&= \begin{pmatrix}
			-r & 0 & 0 \\
			(1+r)/2 & 1 & 0 \\
			0& 0 & 1 \\
		\end{pmatrix}\qquad \text{if $p_1(t_k) = 0$ (collision of type \circled{1}-\circled{2}, or type } \mathfrak{a}),\\
		B &= \begin{pmatrix}
			1 & (1+r)/2 & 0 \\
			0 & -r & 0 \\
			0 & (1+r)/2 & 1 \\
		\end{pmatrix} \qquad\text{if $p_2(t_k) = 0$ (collision of type \circled{2}-\circled{3}, or type } \mathfrak{b}),\\
		C &= \begin{pmatrix}
			1 & 0 & 0 \\
			0 & 1 & (1+r)/2 \\
			0 & 0 & -r \\
		\end{pmatrix} \qquad\text{if $p_3(t_k) = 0$ (collision of type \circled{3}-\circled{4}, or type } \mathfrak{c}).
	\end{split}
	\end{align*}
\noindent
Although the dynamical system $t \mapsto (p(t),q(t))$ clearly depends on the continuous variable $t$, it will be more convenient to consider its \emph{time discretization}, namely:
\begin{align}
\mathbb{N} \ni k \mapsto \left(p(t_k),q(t_k)\right).
\end{align}
In order to simplify the notations, when no ambiguity is possible, we will respectively denote $p(t_k)$ and $q(t_k)$ by $p(k)$ and $q(k)$.

\begin{remark}
The fact that the collisions \circled{1}-\circled{2} and \circled{3}-\circled{4} can take place simultaneously is consistent with the equality $AC = CA$: considering only the velocity variables, the order for which we apply $A$ and $C$ does not matter.
\end{remark}
\noindent
We mentioned that the dynamics is well-posed only when binary collisions take place. Therefore, the dynamics of the particle system is defined recursively for an arbitrary large number of collisions, provided that no multiple collision occurs.\\
Nevertheless, there is no certainty that the dynamics is well-posed on any time interval $[0,T]$, due to the phenomenon of \emph{inelastic collapse}. By definition, such a collapse happens when there exists a time $T_{\infty}\in\R$ such that infinitely many collisions occur on the time interval $[0,T_{\infty}[$. In such a case, the dynamics of the particle system might be ill-posed after $T_\infty$.\\
\newline
Moreover, we will also be interested in the order of collisions that occur. We say that an initial datum $(p(0), q(0))$ \emph{generates} a \emph{collision sequence}, or a \emph{collision pattern} (finite or infinite) $(c_1,c_2,\dots)$ with $c_i \in \{\mathfrak{a,b,c}\}\ \forall i \in \{1,2,3\}$ if these collisions occur in this order, where we allow interchanging $\mathfrak{a}$ and $\mathfrak{c}$ if these two collisions occur at the same time. If the sequence of collisions is eventually periodic (that is, $c_{i+L}=c_i$ for some $L>0$ and for all $i \geq i_0$, for some $i_0$), we call the finite collision sequence $c_{i_0}c_{i_0+1}\dots c_{i_0+L-1}$ the \emph{period} of the collision pattern.



\section{The $\mathfrak{ababcb}$ pattern}
\noindent
In the present section, we will study the pattern $\mathfrak{ababcb}$. What is different about this pattern compared to the ones studied in \cite{CDKK999} is that it is not symmetric (there are more collisions of type $\mathfrak{a}$ than of type $\mathfrak{c}$). As we will see, this pattern is still a possible way to lead to collapse, though the explicit initial datum we compute is unstable, so it will not be observable in numerical simulations.\\ \newline
To find an explicit initial datum that generates this pattern, we will follow the approach of \cite{CDKK999}. We note that since the pattern is periodic, the relative velocities after one iteration of these six collisions will have been multiplied by the matrix $M=BCBABA$:
\begin{align*}
    M=\begin{pmatrix}
(r - 1)^3(r^3 - 3r^2 - 5r - 9)/64 & (r - 1)^2(r^3 - r^2 + 7r + 9))/32 &    (r + 1)^2/4 \\
-r(r^5 - 7r^4 + 10r^3 + 18r^2 + 5r + 5)/32 & -r(r^4 - 4r^3 + 10r^2 + 4r + 5)/16 & -r(r+1)/2\\
(r + 1)^2(r^4 - 12r^3 + 42r^2 - 20r + 5)/64 &        (r - 1)^3(r^2 - 4r - 5)/32 &   (r - 1)^2/4
\end{pmatrix}.
\end{align*}
Since we want to study inelastic collapse, this matrix will be applied to the initial velocities $q(0)$ repeatedly. As is well-known, a repeated application of a matrix to a fixed initial vector will for almost all choices of such an initial vector will provide another vector, such that its direction converge towards the direction of the eigenvector associated to the dominant (largest in absolute value) eigenvalue, assuming the dominant eigenvalue is isolated.\\
Therefore, it makes sense to look for relative velocities that are eigenvectors of $M$, so that after all the six collisions, the resulting velocities $q(t_6)$ at the time $t_6$ of the $6$-th collision have only been rescaled.\\
We will also look for relative distances that have been scaled down by some possibly another scalar after the six collisions:
\begin{align*}
    q(t_6)=\la q(0),\quad p(t_6)=\mu p(0).
\end{align*}
We call these kind of initial data \textit{self-similar}. In order to find distances that satisfy this condition, we will look for a fixed point of a certain function $f$. To be more precise, we can assume without loss of generality that (note that the last collision is $\mathfrak{b}$, so particles \circled{2} and \circled{3} are in contact)
\begin{align*}
    p(0)=\begin{pmatrix}
        1\\0\\x
    \end{pmatrix}.
\end{align*}
Then, after the six collisions $\mathfrak{a,b,a,b,c,b}$, the positions write
\begin{align*}
    p(t_6)=p_1(t_6)\begin{pmatrix}
        1\\
        0\\
        f(x)
    \end{pmatrix},
\end{align*}
where $f$ has the form
\begin{align*}
    f(x)=\frac{p_3(t_6)}{p_1(t_6)}\cdotp
\end{align*}
Finding a fixed point for $f$ will then yield our candidate for $p(0)$.
\begin{theorem}
\label{THEORPatteAsyme}
    The pattern $\mathfrak{ababcb}$ has self-similar initial data if and only if $r<5-2\sqrt{6}\simeq 0.10102$. In that case, these self-similar initial data can be explicitly written as
    \begin{align*}
        q(0)=\begin{pmatrix}
            -1\\
            u\\
            v
        \end{pmatrix},\quad
        p(0)=\begin{pmatrix}
            1\\
            0\\
            x
        \end{pmatrix},
    \end{align*}
    where
    \begin{equation}\label{eq:expressions for uvx}
            \begin{aligned}
        u&=\frac{r(16\la - r^4+ 8r^3 + 2r^2+ 8r- 1)}{2(r + 1)(- r^3 + 6r^2 - r + 4\la)},\\
v&=-\frac{\la(64\la - r^6 + 10r^5 - 23r^4 + 44r^3 +r^2 + 42r - 9) + r^2(20r^4 - 32r^3 + 8r^2 + 4)}{4(r + 1)^2(- r^3 + 6r^2 - r + 4\la)},\\
        x &= \frac{a-d+\sqrt{(a-d)^2+4bc}}{2c}\cdotp
    \end{aligned}
    \end{equation}
Here, $\la$ is either the largest or smallest (positive) real eigenvalue of $M$, and the variables $a,b,c,d$ are given by
\begin{equation}\label{eq:expressions for abcd}
    \begin{aligned}
            a=&-64 r^2 (r^4 - 8 r^3 - 2 r^2 - 8 r + 1)\la+1024r^6,\\
    b=&256r^2(r^2-6r+1)\la^2 - 4r^2(r-1)^2(r-3)^2(r^4 - 8r^3 - 2r^2 - 8 r + 1)\la \\&+ 64r^6(r^4 - 8 r^3 + 18 r^2 + 5),\\
    c=&16(r - 1)^2 (r + 1)^2 (r^2 - 10 r + 1)\la,\\
    d=&-64 (r^4 - 8 r^3 - 2 r^2 - 8 r + 1)\la^2\\
    &+(r^{10} - 18 r^9 + 101 r^8 - 216 r^7 + 66 r^6 + 372 r^5 + 594 r^4 - 24 r^3 + 253 r^2 - 114 r + 9)\la.
    \end{aligned}
\end{equation}
\noindent
However, within the interval $]0,5-2\sqrt{6}[$, these self-similar initial conditions are not stable.
\end{theorem}

\begin{remark}
    The critical threshold of $r$, namely $r_{crit,\mathfrak{ababcb}}=5-2\sqrt{6}$, happens to be also the critical threshold for a different periodic pattern, namely $(\mathfrak{ab})^3(\mathfrak{cb})^3$ (see \cite[Table 1]{CDKK999}).
\end{remark}
\begin{proof}
We proceed in the following steps.
\begin{enumerate}
\item We study the spectrum of $M$. We find that there always exists a positive eigenvalue, on the whole interval $]0,1[$, and, up to $r\simeq 0.11384$, there exist two additional real eigenvalues, which are smaller in absolute value.
    \item We check if there exist positive fixed points for the function $f$ characterizing the relative distances after all $6$ collisions. Below $r=5-2\sqrt{6}\simeq 0.10102$ there exists such a fixed point only for the largest and smallest eigenvalues. Above this threshold, none of the fixed points are positive anymore, hence there are no self-similar initial data.
    \item We check whether the candidate initial data indeed generate the pattern $\mathfrak{ababcb}$ by checking the sign of the velocities and the positivity of the distances.
    \item Within the region where positive fixed points for $f$ exist, we prove instability in the position variable for the upper branch (by considering the modulus of $f'$), and in the velocities on the lower branch (by noting it is not the dominant eigenvalue).
\end{enumerate}
\textbf{Step 1: Studying the spectrum of $M$.}\newline
The characteristic polynomial of $M$ is
\begin{align*}
    \operatorname{cp}(M)&=\la^3-\tfrac{\la^2}{64}(r^6 - 10 r^5 + 23 r^4 - 44 r^3 + 15 r^2 - 74 r + 25)\\
    &+\tfrac{\la}{64} (25 r^6 - 74 r^5 + 15 r^4 - 44 r^3 + 23 r^2 - 10 r + 1)-r^6.
\end{align*}
Since the determinant of $M$ is $r^6$, which is positive, there exists always one positive real eigenvalue: if all eigenvalues are real, then in order for their product (the determinant) to be positive, at least one of the eigenvalues must be positive. In the other case, if there are (necessarily) two non-real eigenvalues that are complex conjugates, the product of two complex conjugates is nonnegative, so that the remaining real eigenvalue must again be positive. To determine when we have three real eigenvalues, we consider the determinant of the characteristic polynomial, which is 
\begin{align*}
    \Delta_{cp}=-\frac{1}{1811939328}(r - 1)^4(r + 1)^6(3r^2 - 14r + 3)^2P_0(r),
\end{align*}
where
\begin{align*}
    P_0(r)=&41r^{10} - 646r^9 + 3733r^8 - 14600r^7 + 42306r^6 - 78052r^5 \\
    &+ 42306r^4 - 14600r^3 + 3733r^2 - 646r + 41.
\end{align*}
We have three real roots of the characteristic polynomial (that is to say, three real eigenvalues) if $\Delta_{cp}<0$, if $\Delta_{cp}=0$ one root has multiplicity at least equal to 2, and if $\Delta_{cp}>0$, two of the roots are not real. Since all factors of $\Delta_{cp}$ except the tenth-degree factor $P_0$ appear with an even power, the sign of $\Delta_{cp}$ can only change when $P_0$ changes sign. Note here that while the factor $3r^2 - 14r + 3$ has one root at $r=(7-2\sqrt{10})/3\simeq 0.22515$ and so two eigenvalues are the same for this particular value of $r$ (and hence there may be a new real positive eigenvalue), one can check that this repeated eigenvalue is negative.\\
We can examine the roots of $P_0$ in $]0,1[$ using for example Sturm's theorem (see \cite[Section 1.4.2]{Pra010}), and it turns out that there is only one root $r_*$ in this interval, and we have $r_*\simeq 0.11384$. Additionally, one can check the signs of $P_0$ around this value to determine that $\Delta_{cp}<0$ for $r<r_*$ and $\Delta_{cp}\geq 0$ for $r>r_*$ (with $\Delta_{cp}>0$ for $r\neq (7-2\sqrt{10})/3$). Therefore there are three real eigenvalues only below $r_*$. The spectrum in the $r-\la$-plane is depicted in Figures \ref{fig:spectrum} and \ref{fig:spectrum zoom}, the latter of which is a zoomed in version of the former. We will divide the spectrum in three parts: the real eigenvalue that always exists (upper branch), the larger of the two eigenvalues that cease to be real around $r_*\simeq 0.11384$ (middle branch) and the smaller one (lower branch).
\begin{figure}[h]
\centering
\begin{minipage}{.5\textwidth}
  \centering
  \includegraphics[width=\linewidth]{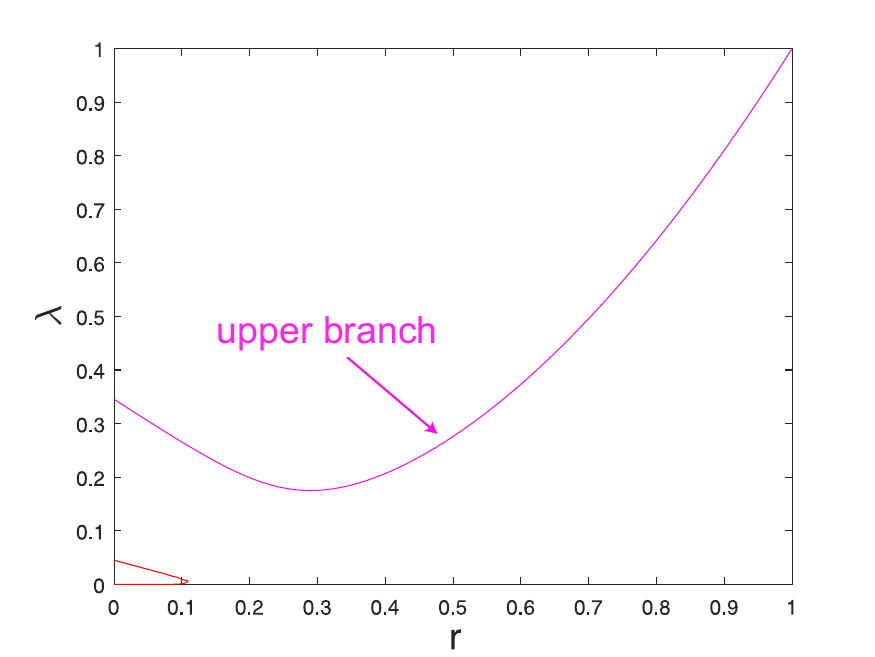}
  \captionof{figure}{The spectrum of $M$.}
  \label{fig:spectrum}
\end{minipage}%
\begin{minipage}{.5\textwidth}
  \centering
  \includegraphics[width=\linewidth]{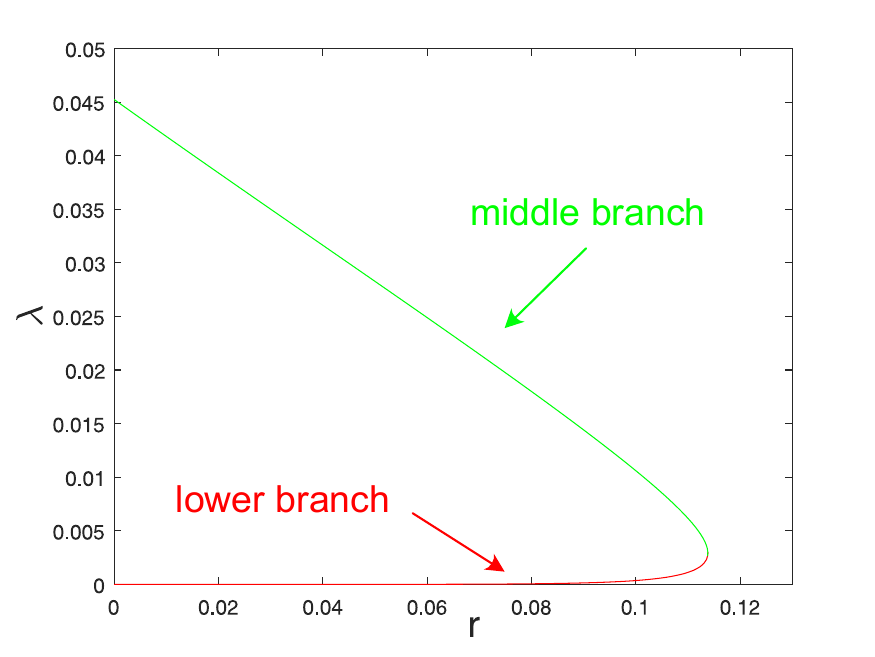}
  \captionof{figure}{The spectrum of $M$ for $0<r<0.13$, and $0 < \lambda < 0.05$.}
  \label{fig:spectrum zoom}
\end{minipage}
\end{figure}
\newline Additionally, we can obtain the expressions for $u$ and $v$ given in \eqref{eq:expressions for uvx} by solving the eigenvalue equation
\begin{align*}
    Mq(0)=\la q(0),
\end{align*}
without explicitly solving for the eigenvalue $\la$. Considering only the first and second components of this eigenvalue equations and solving for $u$ and $v$ shows that they can indeed be expressed as in \eqref{eq:expressions for uvx}.
\\ \newline
\textbf{Step 2: Studying when $f$ has a positive fixed point.}\newline
Writing down the evolution of the distances $p$ (assuming we already know that the collisions will be $\mathfrak{a,b,a,b,c,b}$ in this order) using the explicit expressions for $u$ and $v$ given in \eqref{eq:expressions for uvx}, we find after one period of the pattern has occurred that:
\begin{align*}
    p(t_6)=p_1(t_6)\begin{pmatrix}
        1\\
        0\\
        f(x)
    \end{pmatrix},\text{ where }
    f(x)=\frac{p_3(t_6)}{p_1(t_6)}=\frac{ax+b}{cx+d},
\end{align*}
with $a,b,c,d$ as given in \eqref{eq:expressions for abcd}. Note that these expressions involve an eigenvalue $\la$, which of course depends on $r$ too. \newline Since we are looking for self-similar initial data, we want to study when $f$ has a fixed point, that is, $f(x)=x$. Rearranging, we see that any fixed point of $f$ must be given by
\begin{align*}
     x_\pm = \frac{a-d\pm\sqrt{(a-d)^2+4bc}}{2c}\cdotp
\end{align*}
We now determine the signs of the different quantities in this expression in order to find the sign of $x_\pm$.\\ \newline
\textbf{Step 2.1: Studying the sign of $a-d$.}\newline
In order to determine the sign of
\begin{align*}
    a-d &= \la^2(64r^4 - 512r^3 - 128r^2 - 512r + 64) \\&\hspace{5mm}+ \la(r - 1)^2(- r^8 + 16r^7 - 68r^6 + 64r^5 + 66r^4 + 208r^3 - 116r^2 + 96r - 9)+1024r^6
\end{align*}
on the spectrum of $M$, we consider the common roots of $a-d$ and $\operatorname{cp}(M)$. We can for example compute the reduced Gröbner basis (with respect to the lexicographic ordering $\la>r$; see \cite[Section 6.2]{Pra010} for more details) of the ideal generated by these two polynomials in $r$ and $\la$ to study their intersection points. We find that when $a-d=0$ and $\operatorname{cp}(M)=0$, it must hold that $r$ is a root of
\begin{align}\label{eq: ABABCB a-d necessary roots}
    r^6(r-1)^4(r+1)^4(r^2-10r+1)P_1(r),
\end{align}
where $P_1(r)$ is the following eleventh-degree polynomial:
\begin{align*}
    P_1(r)= 17r^{11} - 435r^{10} + 4363r^9 - 22753r^8 &+ 61530r^7 - 65806r^6 - 27194r^5 \\
    &+ 24494r^4 - 9499r^3 + 3025r^2 - 545r + 35.
\end{align*}
The expression \eqref{eq: ABABCB a-d necessary roots} hence has roots in $]0,1[$ at $r_1:=5-2\sqrt{6}\simeq 0.10102$ and at roots of the eleventh degree polynomial $P_1$. It can be seen that the polynomial $P_1(r)$ (using for example Sturm's theorem again) has three real roots $r_{\leq 0} < r_2 < r_{\geq 1}$, such that only $r_2$ lies in $]0,1[$, and we have $r_2\simeq 0.13145$.
We observe that $a-d$ is a second degree polynomial in $\lambda$, so for every fixed value of $r$ it can have at most two roots, that is, the sign of $a-d$ can change at most on two branches of the eigenvalues. One can check that for $r_1$, this happens on the middle branch, and for $r_2$ the sign also changes (necessarily on the upper branch). Evaluating $a-d$ explicitly for a few values of $r$, we find that the sign of $a-d$ is as illustrated in Figure \ref{fig: sign changes a-d}.
\begin{figure}[h]
\centering
\begin{minipage}{.5\textwidth}
  \centering
    \includegraphics[width=.9\linewidth]{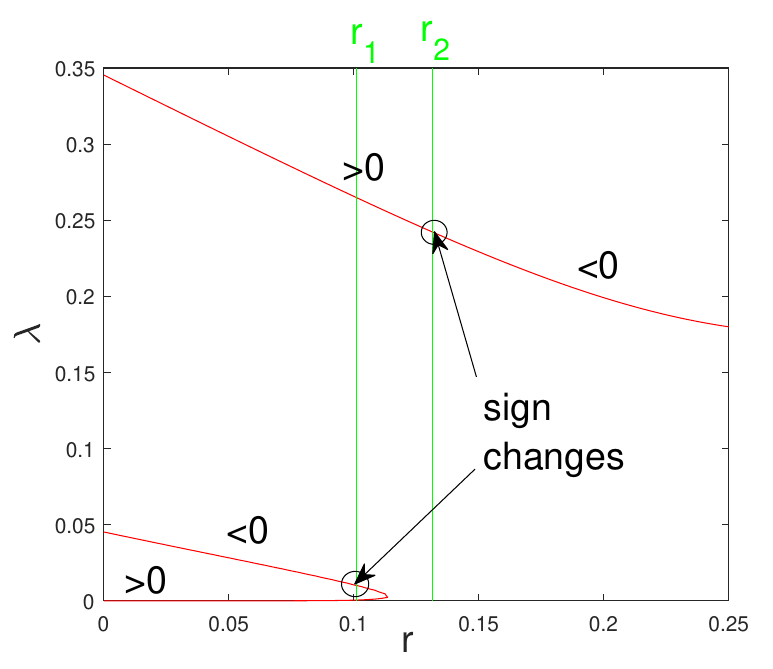}
    \caption{Sign changes of $a-d$.}
    \label{fig: sign changes a-d}
\end{minipage}%
\begin{minipage}{.5\textwidth}
  \centering
    \includegraphics[width=\linewidth]{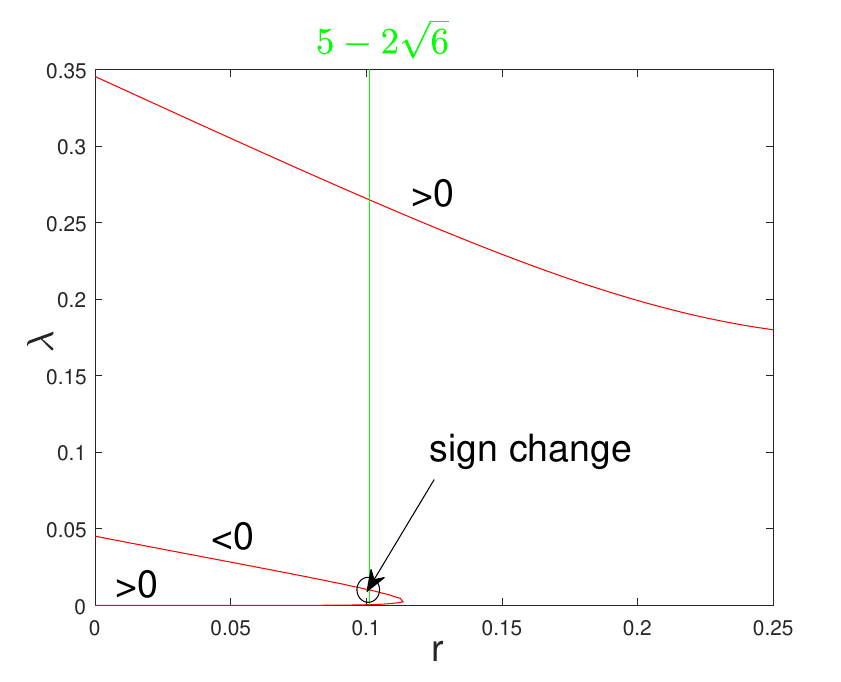}
    \caption{Sign changes of $b$.}
    \label{fig: sign change b}
\end{minipage}
\end{figure}
\\ \newline\textbf{Step 2.2: Studying the sign of $b$.}\newline
The next variable to be examined is
\begin{align*}
    b&=256r^2(r^2-6r+1)\la^2 - 4r^2(r-1)^2(r-3)^2(r^4 - 8r^3 - 2r^2 - 8 r + 1)\la \\&\hspace{5mm}+ 64r^6(r^4 - 8 r^3 + 18 r^2 + 5).
\end{align*}
We proceed as before, showing that for the pairs of $r$ and $\la$ in $\{\operatorname{cp}(M) =0,b=0\}$, it holds that
\begin{align*}
    r^6(r-1)^4(r+1)^4(r^2 - 10r + 1)(5r^4 - 40r^3 + 102r^2 - 40r + 5)=0.
\end{align*}
We can check that for $5-2\sqrt{6}$ (the root of $r^2-10r+1$ on $]0,1[$), the sign of $b$ changes on the lower branches, whereas the fourth-degree factor does not have any real root. By again inserting a few explicit values of $r$, we find that the sign of $b$ is given as shown in Figure \ref{fig: sign change b}.\\ \newline
\textbf{Step 2.3: Studying the sign of $c$.}\newline
Finally, the sign of
\begin{align*}
    c=16(r - 1)^2 (r + 1)^2 (r^2 - 10 r + 1)\la
\end{align*}
is easy to obtain. Since $\la$ has to be positive, the sign of $c$ is the same as the one of $r^2-10r+1$. Hence, we have that $c>0$ for $r<5-2\sqrt{6}$ and $c>0$ for $r>5-2\sqrt{6}$.
\\ \newline
\textbf{Step 2.4: Deducing the sign of the fixed points of $f$.}\newline
After studying the signs of $a-d$, $b$ and $c$ above, we can now determine the sign of the fixed points.\newline For the upper branch, we know that $a-d$, $b$ and $c$ are positive up to $r=5-2\sqrt{6}$, so there exists exactly one positive fixed point, namely $x_+$. At this point however, $c$ changes sign, so between $r=5-2\sqrt{6}$ and $r_2\simeq 0.13145$, the fixed points are not positive (if they are even real). At $r_2$ at latest, the fixed points are not real anymore since $a-d=0$ there and $4bc<0$. One can also check that $(a-d)^2+4bc$ remains negative until $r=1$, so that a positive fixed point on the upper branch can only exist up to $r=5-2\sqrt{6}$.\\
On the middle branch there are no positive fixed points: for $r<5-2\sqrt{6}$, both $a-d$ and $bc$ are negative, whereas $c$ is positive. Then, for $r>5-2\sqrt{6}$, $a-d$ becomes positive while $bc$ stays negative, but since also $c$ changes sign to be negative, the fixed points remain negative.\\
On the lower branch, by the same reasoning as for the upper branch, there is exactly one positive fixed point until $r=5-2\sqrt{6}$, and it is $x_+$.\newline
In conclusion, we observed that there exists one (and only one) positive fixed point on the upper and lower branch below $r=5-2\sqrt{6}$, so that there are two candidates for the initial variable $x$ (again below $r=5-2\sqrt{6}$).
\\
\newline
\textbf{Step 3.1: The velocity inequalities.}
\newline
Our next task is to determine whether the initial datum that we found (the eigenvector of $M$) generates velocities that have the correct sign. That is, we need to check whether after each application of the matrices $A$, $B$ or $C$, the particles that are supposed to collide next are moving towards each other. Note that for this step, we do not yet know if the collisions are the correct ones, but we simply \textit{define} $q(t_1)=Aq(0)$, then $q(t_2) = BAq(0) = Bq(t_1)$, and so on.\newline
Moreover we need to have that after each collision, the particles that just collided are moving away from each other - this will be automatically the case for most of the collisions (if the velocities before are negative, they will be positive afterwards for $r>0$), but we need to check this in particular for $t=0$.\newline
First, we check whether $u$ is positive, so that the initial datum is one that can be attained after a collision of type $\mathfrak{b}$ that just occurred. We note that $u=0$ if and only if $16\la = r^4- 8r^3 - 2r^2- 8r+ 1$, which we can insert in the equation of the characteristic polynomial, so that $u=0$ implies that
\begin{align*}
    (r^2 - 10r + 1)(r^5 - 15r^4 + 66r^3 - 86r^2 - 35r + 5)=0.
\end{align*}
The first factor only has one root in $]0,1[$, which is the familiar value of $5-2\sqrt{6}$. The second factor has exactly one root in $]0,1[$, but it lies at around $r\simeq 0.11376$, which is larger than $5-2\sqrt{6}$. Hence, $u$ does not change its sign up to $5-2\sqrt{6}$, and one can check that it is in fact positive in that interval on the upper and lower branches (as we determined in Step 2.4, on the middle branch there is no positive fixed point for $f$ anyway).\\
In the same way, we can check the other constraints on the velocities, namely:
\begin{itemize} 
    \item $q_2(t_1)=u-\frac{1+r}{2}<0$, 
    \item $q_1(t_2)=\frac{1+r}{2}u-\frac{(r-1)^2}{4}<0$, 
    \item $q_2(t_3)=\frac{(r-1)^2}{4}u-\frac{r^3 -5r^2 -5r +1}{8}<0$, 
    \item $q_3(t_4)=\frac{r^3 - r^2+3r+ 5}{8}u+v-\frac{(r + 1)^2(r^2 - 6r + 5)}{16}<0$, 
    \item $q_2(t_5)=\frac{r^4 - 4 r^3 + 10 r^2 + 4 r + 5}{16}u+\frac{1+r}{2}v-\frac{(r + 1) (r^4 - 8 r^3 + 18 r^2 + 5)}{32}<0$, 
\end{itemize}
to find that all of them are fulfilled on the upper and lower branch - in any case, there is no sign change between $0$ and $r=5-2\sqrt{6}$.\\
\newline\textbf{Step 3.2: The distance inequalities.}
\newline
As a next step, we need to check that the candidate for the initial datum indeed generated the pattern $\mathfrak{ababcb}$. Again, we simply \textit{define} $p(t_1) = p(0)-\frac{p_1(0)}{q_1(0)}q(0)$ (collision of type $\mathfrak{a}$ at time $t_1 = p_1(0)/(-q_1(0))$), $p(t_2) = p(t_1) - \frac{p_2(t_1)}{q_2(t_1)}q(t_1)$ (collision of type $\mathfrak{b}$ at time $t_2 = t_1 + p_2(t_1)/(-q_2(t_1))$), and so on (as if we knew already that the collisions occur in the order we want), and check afterwards that the generated distance vectors have nonnegative (or positive) components. By the previous step, we will then need to check the following:
\begin{itemize}
    \item $p_3(t_i)>0$ for $i=0,\dots,4$,
    \item $p_1(t_5)>0$.
\end{itemize}
We have already examined the first distance inequality (for $t_0=0)$, namely that $p_3(0)=x>0$. Now we observe that since the applications of the matrices $A$ and $B$ can only decrease the third velocity component, $p_3(t_4)>0$ together with $x=p_3(t_0)>0$ implies also that $p_3(t_k)>0$ for $k=1,2,3$. Therefore, we only need to fulfill the following inequalities:
\begin{itemize}
    \item $p_3(t_4)>0$, where
    \begin{align*}
        p_3(t_4)&=x+\frac{(- 64r^2 + 384r - 64)\la^2}{(16r^4 - 128r^3 - 32r^2 - 128r + 16)\la - 256r^4}\\
        &\hspace{5mm} + \frac{(r^8 - 16r^7 + 84r^6 - 192r^5 + 222r^4 - 208r^3 + 196r^2 - 96r + 9)\la}{(16r^4 - 128r^3 - 32r^2 - 128r + 16)\la - 256r^4}\\
        &\hspace{5mm}+\frac{- 16r^8 + 128r^7 - 288r^6 - 80r^4}{(16r^4 - 128r^3 - 32r^2 - 128r + 16)\la - 256r^4}\cdotp
    \end{align*}
    \item $p_1(t_5)>0$, where
    \begin{align*}
        p_1(t_5)&=\frac{\la(r^2 - 1)^2(r^2 - 10r + 1)}{- 64\la^2+4( -r^5 + 9r^4 - 6r^3 + 6r^2 - 9r + 1)\la + 64r^5}x\\
        &\hspace{5mm}+\frac{(64r^4 - 512r^3 - 128r^2 - 512r + 64)\la^2}{1024\la^2 + (64r^5 - 576r^4 + 384r^3 - 384r^2 + 576r - 64)\la - 1024r^5}\\ &+\frac{ (- r^{10} + 18r^9 - 101r^8 + 216r^7 - 66r^6 + 652r^5 - 594r^4 + 24r^3 - 253r^2 + 114r - 9)\la}{1024\la^2 + (64r^5 - 576r^4 + 384r^3 - 384r^2 + 576r - 64)\la - 1024r^5}\\
        &\hspace{5mm}+\frac{- 64r^9 + 512r^8 + 128r^7 + 512r^6 - 64r^5}{1024\la^2 + (64r^5 - 576r^4 + 384r^3 - 384r^2 + 576r - 64)\la - 1024r^5}\cdotp
    \end{align*} 
\end{itemize}
Finally, we also know from Step 2 that the positive fixed point $x$ on the upper and lower branch has the form
\begin{align*}
    x=\frac{a-d+\sqrt{(a-d)^2+4bc}}{2c}\cdotp
\end{align*}
From this we can rearrange each of the equations $p_k(t_i)=0$ into a corresponding polynomial equation again, of which we can find the common roots with the characteristic polynomial as before. In any case, there is no sign change up to $r=5-2\sqrt{6}$.\\
\newline
\textbf{Step 4: Studying if the pattern is stable.} 
\newline
Finally, we examine whether the self-similar initial data that we found has an attracting neighbourhood around it. For the initial data corresponding to the lowest eigenvalue we know that this is not the case since already the velocities are unstable: we know that repeated applications of a matrix will for almost every initial vector converge towards the direction of the eigenvector associated to the dominant (the largest in absolute value) eigenvalue, that is, the eigenvalue on the upper branch. \newline
For the largest eigenvalue (since all eigenvalues up to $r=5-2\sqrt{6}$ are positive, this is also the dominant eigenvalue), we need to examine whether the fixed point $x_+$ for the dynamical system $x_{n+1}=f(x_n)$ is stable. Since $f$ is a Möbius transformation (or homography), this amounts to studying whether the modulus of
\begin{align*}
    q=\frac{cx_++d}{cx_-+d}=\frac{a+d+\sqrt{(a-d)^2+4bc}}{a+d-\sqrt{(a-d)^2+4bc}}
\end{align*}
is larger or smaller than $1$. Namely if $|q|<1$, the fixed point $x_+$ is stable, and if $|q|>1$, it is unstable. The case $q=1$ corresponds to $d^2+4bc=0$, which does not happen below $r=5-2\sqrt{6}$ by Step 2. The case $q=-1$ corresponds to the point where $a+d=0$, which can (by the methods used above) be shown to happen at a root of a $30$-th degree polynomial, namely when
\begin{align*}
    0=r^6(r-1)^6(r+1)^3(r^2-10r+1)P_2(r)P_3(r)
\end{align*}
where
\begin{align*}
    P_2(r)&=25r^8 - 224r^7 + 284r^6 - 416r^5 + 150r^4 - 416r^3 + 284r^2 - 224r + 25,\\
    P_3(r)&=r^5 - 15r^4 + 66r^3 - 86r^2 - 35r + 5.
\end{align*}
The eighth degree polynomial $P_2$ has one root around $r_3\simeq 0.12879$, and $P_3$ has one root around $r_4\simeq 0.11376$, however both lie above $r=5-2\sqrt{6}$.\\
In conclusion, the sign of $|q|-1$ does not change until $5-2\sqrt{6}$, and one can check that $|q|>1$ below this threshold on the upper branch of the eigenvalues, so that this fixed point is unstable. Again note that while $|q|<1$ on the lower branch, the velocity variables are unstable already for the lower branch of the eigenvalues, since the eigenvalue is not the dominant one.\\
\newline
\textbf{Conclusion.}
\newline
In summary, we have found the following: the pattern $\mathfrak{ababcb}$ has feasible self-similar initial data on the upper and lowest branch of the eigenvalues up to $r_{crit}=5-2\sqrt{6}$. On the middle branch there are no positive fixed points for the position map $f$. Above this critical value, there exists no real positive fixed point at all, on any of the branches.\\
Moreover, the pattern in unstable: for the upper branch, the instability comes from the position map $f$ (the fixed point is repelling), whereas for the lowest branch, it comes from the velocities, since they correspond to an eigenvalue of $M$ that is not the largest in absolute value.\\
The proof of Theorem \ref{THEORPatteAsyme} is complete.
\end{proof}
\begin{remark}
    Using the same methods, one can study also the pattern $\mathfrak{abcab}$ and finds that this pattern does not have any self-similar initial datum that generates this pattern. The reason is that for the initial data that are the obvious candidates - that is, the eigenvector of the corresponding matrix $BACBA$ for the velocities and a fixed point for the distance map $f$ for the positions - the collision after the first $\frakB$ collision is $\frakA$, \emph{not} $\frakC$ (and these collisions do not take place at the same time, so that the order really is $\frakA$ and then, after a strictly positive time, $\frakC$). \newline
    Note that this does not mean that the pattern $\mathfrak{abcab}$ is impossible, since the approach used above can only detect self-similar initial data. Interestingly though, for other patterns that we have tested, any pattern that contains this sequence of collisions ($\mathfrak{abcab}$) as a subsequence does not seem to have feasible initial data - which is surprising, since the eigenvectors of the respective matrices and the fixed points for the respective functions $f$ do not have a clear dependence on another. 
\end{remark}

\section{The spherical reduction}
\noindent
As it was already observed in \cite{BeCa999}, the system has clear scalings. It is possible, on the one hand, to rescale all the relative velocities with the same factor, providing the same dynamics, in the sense that the order of the collisions is the same, and even the ratios between consecutive collision times are preserved. The only difference is that all the particles travel their trajectories faster or slower according to this scaling. On the other hand, rescaling with the same factor all the relative positions provides also the same dynamics.
Therefore, we can assume that the vectors
\begin{align}
p(0) = \begin{pmatrix}
p_1(0) \\ p_2(0) \\ p_3(0)
\end{pmatrix}
\hspace{3mm}
\text{and}
\hspace{3mm}
q(0) = \begin{pmatrix}
q_1(0) \\ q_2(0) \\ q_3(0)
\end{pmatrix}
\end{align}
are both normalized. In addition, since we describe the states of the particle system from a collision to the following, one of the three components of $p(0)$ is zero.\\
Let us now implement, for the present case of a system of four particles, the reduction presented in \cite{CoGM995}. In this reference, the authors reduce a system of three one-dimensional inelastic particles to a two-dimensional billiard: the trajectories of the particles, written in terms of the consecutive positions, are piecewise affine functions, taking values in $\mathbb{R}^2$, inscribed in the first quadrant of the plane (this restriction coming from the fact that the relative distances remain positive).\\
\newline
In the case of four particles, we could consider the exact analog of the construction of \cite{CoGM995}, that is, considering a three-dimensional billiard in the first octant of $\mathbb{R}^3$. However, we will show that we can proceed to a two-dimensional reduction. Let us motivate this reduction with a computation.\\
Let us assume that a collision of type $\mathfrak{a}$ just happened, so that in particular the particles \circled{1} and \circled{2} are initially in contact. In that case, the vector $p(0)$ writes:
\begin{align}
p(0) = \begin{pmatrix}
0 \\
\cos\theta\\
\sin\theta,
\end{pmatrix}
\end{align}
for a certain angle $\theta \in\ ]0,\pi/2[$ (the case $\theta = 0$ corresponds to a simultaneous collision involving \circled{1}-\circled{2} and \circled{3}-\circled{4}, while the case $\theta = \pi/2$ corresponds to a triple collision). We can now decompose the vector of the initial relative velocities, along the tangent plane to the unit sphere at $(0,\cos\theta,\sin\theta)$, and its orthogonal. Let us observe that an orthonormal basis of this tangent plane is given by the two vectors $e_x = (1,0,0)$ and $e_\theta = (0,-\sin\theta,\cos\theta)$. In the limit case $\theta = 0$, the vector of the initial positions becomes $(0,1,0)$, so that $\{(1,0,0),(0,1,0),(0,0,1)\}$, corresponding to $\{e_x,p(0),e_\theta\}$, is a direct orthonormal basis of the space. The second limit case is $\theta = \pi/2$, providing $\{(1,0,0),(0,0,1),(0,-1,0)\}$, which is again a direct basis. Therefore, it is natural to measure the angles in the tangent plane, where the direct orientation is chosen from $e_\theta$ to $e_x$. Assuming that the vector $q(0)$ of the initial velocities lies in the tangent plane, it can be written as:
\begin{align}
\cos\varphi \cdot e_\theta + \sin \varphi \cdot e_x.
\end{align}
It is clear that the value of the angle $\varphi$ is prescribing which collision will follow from the initial configuration: if $\varphi \in\ ]0,\pi/2[$, the trajectory $p(t) = p(0) + t q(0)$ will intersect the plane $y=0$ before the plane $z=0$, which means that the next collision to take place involves the particles \circled{2} and \circled{3}. On the opposite, if $\varphi \in\ ]\pi/2,\pi[$, the trajectory $p(t)$ will intersect first the plane $z=0$, implying that the next collision involves the particles \circled{3} and \circled{4}.\\
Of course, in general the vector of the initial velocities has no reason to belong to the tangent plane to the sphere at $p(0)$. However, let us observe that \emph{only} this angle $\varphi$ is prescribing the next collision when the vector $p(0)$ of the initial relative positions is fixed, in the sense that the component of $q(0)$ that lies in the orthogonal to the tangent plane to the unit sphere plays no role in this matter.\\
\newline
We will show in this section a result that is much stronger, and which is somehow surprising. According to the description we followed, representing the initial positions as a unit vector in the plane $x=0$, and the initial velocities as another unit vector, we define in this way a plane, generated by these two vectors. After this initial configuration, when a first collision takes place, another pair of particles is now in contact, and we can represent the vector of the new positions as (assuming for instance that the next collision is of type $\mathfrak{b}$)
\begin{align}
p(1) = r\begin{pmatrix} \sin\theta_1 \\ 0 \\ \cos\theta_1 \end{pmatrix}.
\end{align}
In the same way, the vector of the velocities, right after this collision of type $\mathfrak{b}$ becomes $q(1) = B q(0)$. These two vectors $p(1)$ and $q(1)$ define a new plane, and we can repeat the procedure, for all the collisions taking place in the future.

\begin{figure}[h]
\centering
    \includegraphics[trim = 0cm 0cm 0cm 0cm, width=1\linewidth]{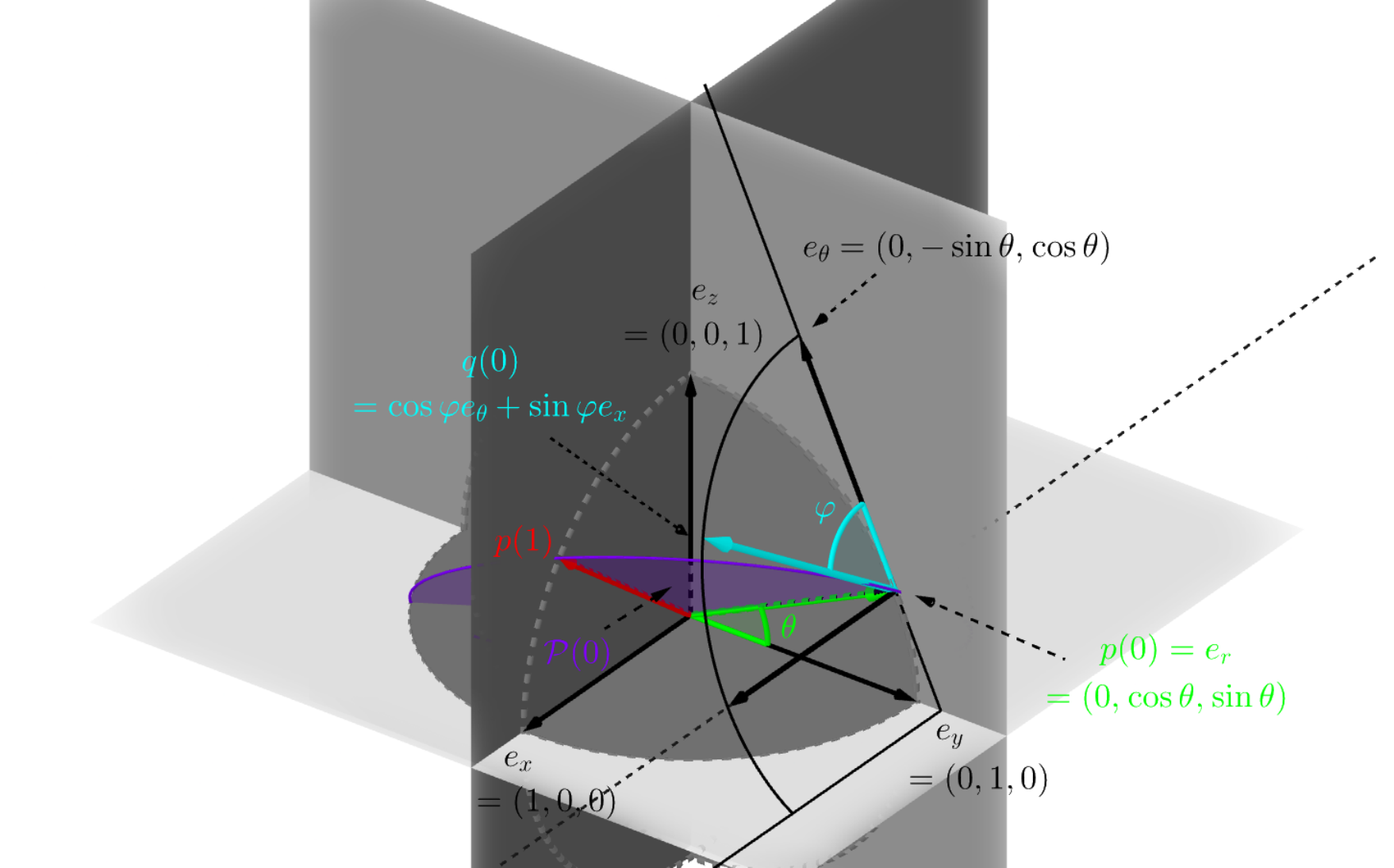}
    \caption{Representation in the first octant of $\mathbb{R}^3$ of a trajectory assuming that $p(0)$ and $q(0)$ are normalized, and that the initial configuration corresponds to a collision of type $\mathfrak{a}$. The two vectors $p(0)$ and $q(0)$ span the plane $\mathcal{P}(0)$ (in purple on the figure). The intersection between $\mathcal{P}(0)$ and $y=0$ (because $\varphi \in\ ]0,\pi/2[$ on the figure, so the next collision is of type $\mathfrak{b}$) defines the direction of $p(1)$. To determine $\mathcal{P}(1)$, one computes $Bq(0)$, then projects this vector on $p(1)^\perp$ and finally normalizes it to obtain $q(1)$, so that $\mathcal{P}(1) = \text{Span}\left(p(1),q(1)\right)$. One can then repeat the process forever to define the sequence of planes $\left(\mathcal{P}(k)\right)_k$.}
    \label{fig:IllusSpherReduc}
\end{figure}

\noindent
Therefore, it is always possible to construct from a sequence of collisions the associated sequence of planes. What we will show is that, for a given plane $\mathcal{P}(0)$ that contains $p(0)$ and $q(0)$, any other sequence of collisions starting from $p'(0)$ and $q'(0)$ such that $\text{Span}(p'(0),q'(0)) = \text{Span}(p(0),q(0)) = \mathcal{P}(0)$ is associated to the same sequence of planes as the sequence of collisions starting from $p(0)$ and $q(0)$. This result, which is one of the main theorems of the present article, is the object of the following section.\\
An illustration of the first step of the construction is depicted in Figure \ref{fig:IllusSpherReduc}.\footnote{The reader may also find an interactive $3$d version of Figure \ref{fig:IllusSpherReduc} following the link: https://www.geogebra.org/m/cpddxb6a}

\subsection{The spherical reduction theorem}

\begin{theorem}[Spherical billiard reduction]
\label{THEORSpherical_Reduc}
Let $0 < r < 1$ be a positive real number. Let us consider two trajectories of four one-dimensional inelastic particles colliding with restitution coefficient $r$, respectively denoted by $\left(p(k),q(k)\right)_{k \in \mathbb{N}}$ and $\left(p'(k),q'(k)\right)_{k \in \mathbb{N}}$. Let us assume that both trajectories present infinitely many collisions, starting from the two respective initial configurations $\left(p(0),q(0)\right)$ and $\left(p'(0),q'(0)\right) \in \mathbb{S}^2\times \mathbb{S}^2$, such that the same pairs of particles are in contact initially, that is:
\begin{align}
\forall j \in \{1,2,3\}, \hspace{2mm} p_j(0) = 0 \Leftrightarrow p'_j(0),
\end{align}
and such that the initial configurations of the two trajectories generate the same plane, that is:
\begin{align}
\text{Span}\left(p(0),q(0)\right) = \text{Span}\left(p'(0),q'(0)\right).
\end{align}
Then, the two trajectories generate the same planes at each collision, that is:
\begin{align}
\forall k \in \mathbb{N}^*, \hspace{2mm} \text{Span}\left(p(k),q(k)\right) = \text{Span}\left(p'(k),q'(k)\right).
\end{align}
\end{theorem}

\begin{proof}
By symmetry, it is enough to consider only two types of initial data: on the one hand when the particles \circled{1} and \circled{2} are in contact, and on the other hand when \circled{2} and \circled{3} are in contact. Indeed, the case when \circled{3} and \circled{4} are in contact corresponds by symmetry to the first case.\\
If \circled{1} and \circled{2} are initally in contact, two types of collision can take place in the future: a collision of type $\mathfrak{b}$, between \circled{2} and \circled{3}, and a collision of type $\mathfrak{c}$, between \circled{3} and \circled{4}. If \circled{2} and \circled{3} are initially in contact, by symmetry, only one collision has to be described in the future: a collision of type $\mathfrak{c}$, between \circled{3} and \circled{4} (by symmetry, we would also cover the case of a collision between \circled{1} and \circled{2}).\\
We will study these three cases separately.\\
\newline
$\underline{\mathfrak{a}\mathfrak{b}}$. Let us assume that \circled{1} and \circled{2} are initially in contact, and that the next collision to follow involves \circled{2} and \circled{3}. We describe then the sequence of the two collisions $\mathfrak{a} \mathfrak{b}$.\\
In this case, the initial configuration $p(0)$, after renormalization, can be written as:
\begin{align}
\label{EQUATInitiPositCol_a}
p(0) = \begin{pmatrix}
0 \\
\cos\theta \\
\sin\theta
\end{pmatrix},
\end{align}
for a certain angle $\theta \in [0,\pi/2[$. Note that we have to exclude the case $\theta = \pi/2$, because then the vector of the initial positions would have the form $\begin{pmatrix} 0 \\ 0 \\ 1 \end{pmatrix}$, meaning that \circled{1}, \circled{2} and \circled{3} are all initially in contact, which corresponds to an ill-posed configuration of the particle system.\\
As for the vector of the initial velocities $q(0)$, in order to simplify the computations, we will not assume that it is normalized, but we will assume that the tangential component of $q(0)$ is normalized. More precisely, the tangent plane of the unit sphere at $p(0)$ is
\begin{align}
\text{Span}(e_\theta,e_x) = \text{Span}\left( \begin{pmatrix} 0 \\ -\sin\theta \\ \cos\theta \end{pmatrix},\begin{pmatrix} 1 \\ 0 \\ 0 \end{pmatrix}\right),
\end{align}
so that the (normalized) tangential part of $q(0)$ writes:
\begin{align}
\cos\varphi \cdot e_\theta + \sin\varphi \cdot e_x = \cos\varphi \begin{pmatrix} 0 \\ -\sin\theta \\ \cos\theta \end{pmatrix} + \sin\varphi \begin{pmatrix} 1 \\ 0 \\ 0 \end{pmatrix}
\end{align}
for a certain angle $\varphi \in\ ]0,\pi[$, and so, there exists $\lambda \in \mathbb{R}$ such that the vector of the initial velocities $q(0)$ writes:
\begin{align}
\label{EQUATInitiVelocCol_a}
q(0) = \cos\varphi \cdot e_\theta + \sin\varphi \cdot e_x + \lambda e_r = \begin{pmatrix} \sin\varphi \\ -\cos\varphi\sin\theta + \lambda \cos\theta \\ \cos\varphi\cos\theta + \lambda \sin\theta \end{pmatrix}.
\end{align}
To emphasize that the different vectors of initial velocities depend on $\lambda$, we will denote them by $q_\lambda(0)$ instead of $q(0)$.\\
In this way, we are indeed considering all the initial configurations which generate the same plane $\mathcal{P}(0) = \text{Span}\left(p(0),q(0)\right)$. Indeed, the trajectories start all from $p(0)$, and since their respective vectors of initial velocities are $q_\lambda(0)$, each trajectory belongs to the respective planes $\mathcal{P}_\lambda = \text{Span}\left(p(0),q_\lambda(0)\right)$. But since $e_r$ and $p(0)$ are colinear, the choice of $\lambda$ has no impact, so all the $\mathcal{P}_\lambda$ are the same.\\
By assumption the next collision is of type $\mathfrak{b}$, so that $\varphi \in\ ]0,\pi/2[$ (and observe also that, then $\cos\varphi > 0$). To see this, we simply compute the collision times $t_\mathfrak{b}$ and $t_\mathfrak{c}$, obtained if the next collision is of type $\mathfrak{b}$ or $\mathfrak{c}$. We find:
\begin{align}
t_\mathfrak{b} = \frac{\cos\theta}{\sin\theta\cos\varphi - \lambda\cos\theta} = \frac{1}{\tan\theta\cos\varphi - \lambda} \hspace{5mm} \text{and} \hspace{5mm} t_\mathfrak{c} = \frac{\sin\theta}{-\cos\theta\cos\varphi -\lambda\sin\theta} = \frac{1}{-\frac{\cos\varphi}{\tan\theta} - \lambda},
\end{align}
so that
\begin{align}
\label{EQUATCompaColliTimesAfterColl_a}
t_\mathfrak{c} - t_\mathfrak{b} = \frac{\cos\varphi\tan\theta + \cos\varphi/\tan\theta}{(-\cos\varphi/\tan\theta -\lambda)(\tan\theta\cos\varphi - \lambda)},
\end{align}
and so, if both $t_\mathfrak{b}$ and $t_\mathfrak{c}$ are positive, the sign of $t_\mathfrak{c} - t_\mathfrak{b}$ is given by the sign of $\cos\varphi$.\\
Nevertheless, let us make the following observations. On the one hand, the signs of the collision times $t_\mathfrak{b}$ and $t_\mathfrak{c}$ depend also on $\lambda$. On the other hand, when $\cos\varphi > 0$, if $t_\mathfrak{b}$ is negative, we have $0 < \tan\theta \cos\varphi < \lambda$, so that $-\cos\varphi/\tan\theta < \lambda$, and so $t_\mathfrak{c}$ is negative as well. Therefore, we have the following cases: in the case when $\cos\varphi > 0$, either $t_\mathfrak{b}$ and $t_\mathfrak{c}$ are both negative, or $t_\mathfrak{b}$ is positive, in which case either $t_\mathfrak{c}$ is negative, or $t_\mathfrak{c}$ is positive but then smaller than $t_\mathfrak{b}$. We can then conclude that if some collision takes place in the future after a collision of type $\mathfrak{a}$, in the case $\cos\varphi > 0$, the first collision has to be of type $\mathfrak{b}$.\\
The opposite conclusion follows in the case $\cos\varphi < 0$: a collision of type $\mathfrak{a}$, if followed by any, has to be followed by a collision of type $\mathfrak{c}$.\\
Let us now go back to the description of a trajectory such that a collision of type $\mathfrak{b}$ follows a collision of type $\mathfrak{a}$. Until this collision of type $\mathfrak{b}$, the trajectory of the particle system writes:
\begin{align}
\label{EQUATTraje_fromCol_a}
p(t) = p(0) + tq_\lambda(0) = \begin{pmatrix}
0 \\
\cos\theta \\
\sin\theta
\end{pmatrix} + t \begin{pmatrix} \sin\varphi \\ -\cos\varphi\sin\theta + \lambda \cos\theta \\ \cos\varphi\cos\theta + \lambda \sin\theta \end{pmatrix}.
\end{align}
The collision $\mathfrak{b}$ takes place at a time that we denote by $t_\lambda$, characterized by the condition $p_2(t_\lambda) = 0$, that is:
\begin{align}
\cos\theta + t_\lambda\left( \lambda\cos\theta - \cos\varphi\sin\theta \right) = 0,
\end{align}
or again
\begin{align}
t_\lambda = \frac{\cos\theta}{\cos\varphi\sin\theta - \lambda\cos\theta} \cdotp
\end{align}
Therefore, the position $p_\lambda(1)$, at the time $t_\lambda$, is:
\begin{align}
p_\lambda(1) &= \begin{pmatrix}
0 \\
\cos\theta \\
\sin\theta
\end{pmatrix} + \frac{\cos\theta}{\cos\varphi\sin\theta - \lambda\cos\theta} \begin{pmatrix} \sin\varphi \\ -\cos\varphi\sin\theta + \lambda \cos\theta \\ \cos\varphi\cos\theta + \lambda \sin\theta \end{pmatrix} \nonumber\\
&= \begin{pmatrix} \frac{\sin\varphi\cos\theta}{\cos\varphi\sin\theta-\lambda\cos\theta} \\ 0 \\ \frac{\cos\varphi\sin^2\theta - \lambda\cos\theta\sin\theta + \cos\varphi\cos^2\theta + \lambda \cos\theta\sin\theta}{\cos\varphi\sin\theta - \lambda\cos\theta} \end{pmatrix} \nonumber\\
&= \frac{1}{\cos\varphi\sin\theta - \lambda\cos\theta} \begin{pmatrix} \sin\varphi \cos\theta \\ 0 \\ \cos\varphi \end{pmatrix}.
\end{align}
The last computation shows that all the $p_\lambda(1)$ are colinear, for all the choices of $\lambda$, which is consistent with the fact that all the $\mathcal{P}_\lambda(0)$ are the same, so their respective intersections with the plane $y=0$ (which contain then $p_\lambda(1)$) are colinear.\\
When the collision $\mathfrak{b}$ takes place, the vector of the velocities $q(0)$ is modified and becomes $q_\lambda(1) = B q_\lambda(0)$. We have to check that the planes $\mathcal{P}_\lambda(1) = \text{Span}\left(p_\lambda(1),q_\lambda(1)\right)$ all coincide. In other words, we want to check that the direction of the vector:
\begin{align}
p_\lambda(1) \wedge q_\lambda(1)
\end{align}
does not depend on $\lambda$. If we denote by $p(1)$ and $q(1)$ the vectors of positions and velocities respectively, both obtained when $\lambda = 0$, we have to show, equivalently, that
\begin{align}
p_\lambda(1) \wedge q_\lambda(1) \hspace{3mm} \text{and} \hspace{3mm} p(1)\wedge q(1)
\end{align}
are colinear, for all $\lambda$. But since $p_\lambda(1)$ and $p(1)$ are colinear, this is equivalent to show that $q_\lambda(1)$ belongs to the plane generated by $p(1)$ and $q(1)$. This last condition is equivalent to:
\begin{align}
q_\lambda(1) \cdot \left[ p(1)\wedge q(1) \right] = 0,
\end{align}
which, by linearity, is equivalent to:
\begin{align}
(B e_r) \cdot \left[ p(1)\wedge q(1) \right] = 0.
\end{align}
This verification is obtained by the following straightforward computation. Let us observe that we can simplify the computation by omitting the denominator of the expression of $p_\lambda(1)$, which will play no role in the colinearity we want to check. $p(1) \wedge q(1)$ is then colinear to the following vector, and we have:
\begin{align}
\begin{pmatrix} \sin\varphi \cos\theta \\ 0 \\ \cos\varphi \end{pmatrix} \wedge (B q(0)) &= \begin{pmatrix} \sin\varphi \cos\theta \\ 0 \\ \cos\varphi \end{pmatrix} \wedge \left( \begin{pmatrix} 1 & (1+r)/2 & 0 \\ 0 & -r & 0 \\ 0 & (1+r)/2 & 1 \end{pmatrix} \begin{pmatrix} \sin\varphi \\ -\cos\varphi\sin\theta \\ \cos\varphi\cos\theta \end{pmatrix} \right) \nonumber\\
&= \begin{pmatrix} \sin\varphi \cos\theta \\ 0 \\ \cos\varphi \end{pmatrix} \wedge \begin{pmatrix} \sin\varphi - \frac{1+r}{2}\cos\varphi\sin\theta \\ r \cos\varphi\sin\theta \\ -\frac{1+r}{2}\cos\varphi\sin\theta + \cos\varphi\cos\theta \end{pmatrix} \nonumber\\
&= \begin{pmatrix}
- r \cos^2\varphi \sin\theta \\
\cos\varphi\sin\varphi - \frac{1+r}{2}\cos^2\varphi\sin\theta + \frac{1+r}{2}\cos\varphi\sin\varphi\cos\theta\sin\theta - \cos\varphi\sin\varphi\cos^2\theta \\
r \cos\varphi\sin\varphi\cos\theta\sin\theta
\end{pmatrix}.
\end{align}
Since $Be_r$ is equal to:
\begin{align}
B e_r = \begin{pmatrix} 1 & (1+r)/2 & 0 \\ 0 & -r & 0 \\ 0 & (1+r)/2 & 1 \end{pmatrix} \begin{pmatrix} 0 \\ \cos\theta \\ \sin\theta \end{pmatrix} &= \begin{pmatrix} \frac{1+r}{2}\cos\theta \\ -r\cos\theta \\ \frac{1+r}{2}\cos\theta + \sin\theta \end{pmatrix},
\end{align}
we find:
\begin{align}
&\hspace{-10mm}(B e_r) \cdot \left[ (\cos\varphi\sin\theta) p(1)\wedge q(1) \right] \nonumber\\
&= \begin{pmatrix} \frac{1+r}{2}\cos\theta \\ -r\cos\theta \\ \frac{1+r}{2}\cos\theta + \sin\theta \end{pmatrix} \nonumber\\
&\hspace{5mm}\cdot \begin{pmatrix}
- r \cos^2\varphi \sin\theta \\
\cos\varphi\sin\varphi - \frac{1+r}{2}\cos^2\varphi\sin\theta + \frac{1+r}{2}\cos\varphi\sin\varphi\cos\theta\sin\theta - \cos\varphi\sin\varphi\cos^2\theta \\
r \cos\varphi\sin\varphi\cos\theta\sin\theta
\end{pmatrix} \nonumber\\
&= \left(-r\frac{1+r}{2}\cos^2\varphi\cos\theta\sin\theta\right)
+ \big( -r \cos\varphi\sin\varphi\cos\theta + r\frac{1+r}{2}\cos^2\varphi\cos\theta\sin\theta \nonumber\\
&\hspace{5mm}- r\frac{1+r}{2}\cos\varphi\sin\varphi\cos^2\theta\sin\theta+r\cos\varphi\sin\varphi\cos^3\theta \big) \nonumber\\
&\hspace{5mm}+ \left( r\frac{1+r}{2}\cos\varphi\sin\varphi\cos^2\theta\sin\theta + r \cos\varphi\sin\varphi\cos\theta\sin^2\theta \right) \nonumber\\
&= -r \cos\varphi\sin\varphi\cos\theta + r \cos\varphi\sin\varphi\cos\theta (\cos^2\theta + \sin^2\theta) = 0.
\end{align}
We have shown that any trajectory starting from a plane $\mathcal{P}(0)$ with a collision of type $\mathfrak{a}$, immediately followed by a collision of type $\mathfrak{b}$, is contained in a plane $\mathcal{P}(1)$ after this second collision, and this plane $\mathcal{P}(1)$ depends only on $\mathcal{P}(0)$.\\
\newline
$\underline{\mathfrak{a}\mathfrak{c}}$. Let us assume that \circled{1} and \circled{2} are initially in contact, and that the next collision to follow involves \circled{3} and \circled{4}. We describe then the sequence of the two collisions $\mathfrak{a} \mathfrak{c}$.\\
In this case, the initial configuration of the system is the same as in the previous case, so that $p(0)$ and $q(0)$ are given by the expressions \eqref{EQUATInitiPositCol_a} and \eqref{EQUATInitiVelocCol_a}, while the trajectory until the first collision $\mathfrak{c}$ is written as in \eqref{EQUATTraje_fromCol_a}. However, assuming that the next collision is of type $\mathfrak{c}$ is equivalent to assume $\varphi \in\ ]\pi/2,\pi[$ (indeed, observe that in this case we need to have $\cos\varphi < 0$, see \eqref{EQUATCompaColliTimesAfterColl_a}). The collision time $t_\lambda$ is now defined such that $p_3(t_\lambda) = 0$, so that we have:
\begin{align}
\sin \theta + t_\lambda \left( \cos\varphi \cos\theta + \lambda \sin\theta \right) = 0,
\end{align}
that is:
\begin{align}
t_\lambda = \frac{-\sin\theta}{\cos\varphi \cos\theta + \lambda\sin\theta} \cdotp
\end{align}
Therefore, the vector of positions of the particle system right after the collision $\mathfrak{c}$ becomes:
\begin{align}
p_\lambda(1) &= \begin{pmatrix} 0 \\ \cos\theta \\ \sin\theta \end{pmatrix} + \frac{-\sin\theta}{\cos\varphi \cos\theta + \lambda\sin\theta} \begin{pmatrix} \sin\varphi \\ -\cos\varphi\sin\theta + \lambda \cos\theta \\ \cos\varphi\cos\theta + \lambda \sin\theta \end{pmatrix} \nonumber\\
&= \begin{pmatrix}
\frac{-\sin\varphi \sin\theta}{\cos\varphi\cos\theta + \lambda\sin\theta} \\
\frac{\cos\varphi\cos^2\theta + \lambda \cos\theta\sin\theta + \cos\varphi\sin^2\theta - \lambda \cos\theta \sin\theta}{\cos\varphi\cos\theta + \lambda\sin\theta}\\
0
\end{pmatrix} \nonumber\\
&= \frac{1}{\cos\varphi\cos\theta + \lambda\sin\theta} \begin{pmatrix}
-\sin\varphi \sin\theta \\
\cos\varphi \\
0
\end{pmatrix}.
\end{align}
As in the previous step, the plane $\mathcal{P}_\lambda(1) = \text{Span}\left(p_\lambda(1),q_\lambda(1)\right)$ will not depend on $\lambda$ if and only if:
\begin{align}
(C e_r) \cdot \left[p(1) \wedge q(1)\right] = 0.
\end{align}
In the present case, $p(1) \wedge q(1)$ is colinear to
\begin{align}
\begin{pmatrix} -\sin\varphi\sin\theta \\ \cos\theta \\ 0 \end{pmatrix} \wedge \left( C q(0) \right) &= \begin{pmatrix} -\sin\varphi\sin\theta \\ \cos\varphi \\ 0 \end{pmatrix} \wedge \left( \begin{pmatrix} 1 & 0 & 0 \\ 0 & 1 & (1+r)/2 \\ 0 & 0 & -r \end{pmatrix} \begin{pmatrix} \sin\varphi \\ -\cos\varphi\sin\theta \\ \cos\varphi\cos\theta \end{pmatrix} \right) \nonumber\\
&= \begin{pmatrix} -\sin\varphi\sin\theta \\ \cos\varphi \\ 0 \end{pmatrix} \wedge \begin{pmatrix}
\sin\varphi \\
-\cos\varphi\sin\theta + \frac{1+r}{2}\cos\varphi\cos\theta \\
-r\cos\varphi\cos\theta 
\end{pmatrix}\nonumber\\
&= \begin{pmatrix}
- r \cos^2\varphi \cos\theta \\
-r \cos\varphi\sin\varphi \cos\theta \sin\theta \\
\cos\varphi \sin\varphi \sin^2\theta - \frac{1+r}{2} \cos\varphi\sin\varphi \cos\theta \sin\theta - \cos\varphi\sin\varphi
\end{pmatrix},
\end{align}
and since $C e_r$ is equal to:
\begin{align}
C e_r = \begin{pmatrix} 1 & 0 & 0 \\ 0 & 1 & (1+r)/2 \\ 0 & 0 & -r \end{pmatrix} \begin{pmatrix} 0 \\ \cos\theta \\ \sin\theta \end{pmatrix} = \begin{pmatrix} 0 \\ \cos\theta + \frac{1+r}{2}\sin\theta \\ -r\sin\theta \end{pmatrix}
\end{align}
in the end we obtain:
\begin{align}
&\hspace{-10mm} (C e_r) \cdot \left[(\cos\varphi\cos\theta)p(1) \wedge q(1)\right] \nonumber\\
&= \begin{pmatrix} 0 \\ \cos\theta + \frac{1+r}{2}\sin\theta \\ -r\sin\theta \end{pmatrix} \nonumber\\
&\hspace{5mm} \cdot \begin{pmatrix}
- r \cos^2\varphi \cos\theta \\
-r \cos\varphi\sin\varphi \cos\theta \sin\theta \\
\cos\varphi \sin\varphi \sin^2\theta - \frac{1+r}{2} \cos\varphi\sin\varphi \cos\theta \sin\theta - \cos\varphi\sin\varphi
\end{pmatrix} \nonumber\\
&= \left( -r\cos\varphi\sin\varphi\cos^2\theta\sin\theta - r \frac{1+r}{2}\cos\varphi\sin\varphi\cos\theta\sin^2\theta \right) \nonumber\\
&\hspace{5mm} + \left( -r\cos\varphi\sin\varphi\sin^3\theta + r \frac{1+r}{2}\cos\varphi\sin\varphi\cos\theta\sin^2\theta + r \cos\varphi\sin\varphi\sin\theta \right) \nonumber\\
&= -r\cos\varphi\sin\varphi(\cos^2\theta + \sin^2\theta)\sin\theta + r \cos\varphi\sin\varphi \sin\theta = 0.
\end{align}
Therefore, in this case also we have shown that any trajectory starting from a plane $\mathcal{P}(0)$ with a collision of type $\mathfrak{a}$, immediately followed by a collision of type $\mathfrak{c}$, is contained in a plane $\mathcal{P}(1)$ after this second collision, and this plane $\mathcal{P}(1)$ depends only on $\mathcal{P}(0)$.\\
\newline
$\underline{\mathfrak{b}\mathfrak{c}}$. In this last case, we assume that the particles \circled{2} and \circled{3} are initially in contact, and that the next collision to take place involves the particles \circled{3} and \circled{4}. We will describe the sequence of collisions $\mathfrak{b}\mathfrak{c}$.\\
In this case, the vector of the initial positions $p(0)$ is:
\begin{align}
p(0) = \begin{pmatrix}
\sin\theta \\ 0 \\ \cos\theta
\end{pmatrix},
\end{align}
for an angle $\theta \in\ ]0,\pi/2[$. Let us observe here that we have to exclude the two cases $\theta = 0$ and $\theta = \pi/2$, which would imply respectively that \circled{1} and \circled{2}, and \circled{3} and \circled{4} are in contact, yielding in both cases an ill-posed dynamics due to a triple collision.\\
As for the two previous cases, let us consider the vector of the initial velocities such that its tangential component, with respect to the tangent plane of the unit sphere at $p(0)$, is normalized. In the present case, the tangent space of the unit sphere at $p(0)$ writes:
\begin{align}
\text{Span}\left( \begin{pmatrix} \cos\theta \\ 0 \\ -\sin\theta \end{pmatrix},\begin{pmatrix} 0 \\ 1 \\ 0 \end{pmatrix} \right),
\end{align}
and so, assuming that the tangential component of $q_\lambda(0)$ is normalized implies that there exists $\lambda \in \mathbb{R}$ and an angle $\varphi \in\ ]0,\pi[$ such that:
\begin{align}
q_\lambda(0) = \cos\varphi \cdot e_\theta + \sin\varphi \cdot e_y + \lambda e_r = \begin{pmatrix}
\cos\varphi\cos\theta + \lambda\sin\theta \\
\sin\varphi \\
-\cos\varphi\sin\theta + \lambda\cos\theta
\end{pmatrix}.
\end{align}
Since we assumed that the next collision to take place is involving \circled{3} and \circled{4}, we have $\varphi \in\ ]0,\pi/2[$, because the collision times $t_\mathfrak{c}$ and $t_\mathfrak{a}$ write respectively (assuming that the first collision to follow is, respectively, of type $\mathfrak{c}$ or $\mathfrak{a}$):
\begin{align}
t_\mathfrak{c} = \frac{\cos\theta}{\sin\theta\cos\varphi - \lambda\cos\theta} = \frac{1}{\tan\theta\cos\varphi - \lambda} \hspace{5mm} \text{and} \hspace{5mm} t_\mathfrak{a} = \frac{\sin\theta}{-\cos\theta\cos\varphi - \lambda\sin\theta} = \frac{1}{-\frac{\cos\varphi}{\tan\theta} - \lambda},
\end{align}
so that $t_\mathfrak{c}$ is always the smallest positive collision time if and only if $\cos\varphi > 0$.\\
Then, the collision time $t_\lambda$ is defined such that $p_3(t_\lambda) = p_3(0) + t_\lambda q_{3,\lambda}(0) = 0$, that is:
\begin{align}
t_\lambda = \frac{\cos\theta}{\cos\varphi\sin\theta - \lambda\cos\theta} \cdotp
\end{align}
The vector $p_\lambda(1)$ of the positions at the collision time $t_\lambda$ becomes:
\begin{align}
p_\lambda(1) &= \begin{pmatrix} \sin\theta \\ 0 \\ \cos\theta \end{pmatrix} + \frac{\cos\theta}{\cos\varphi\sin\theta - \lambda\cos\theta} \begin{pmatrix}
\cos\varphi\cos\theta + \lambda\sin\theta \\
\sin\varphi \\
-\cos\varphi\sin\theta + \lambda\cos\theta
\end{pmatrix} \nonumber\\
&= \begin{pmatrix}
\frac{\cos\varphi\sin^2\theta - \lambda\cos\theta\sin\theta + \cos\varphi \cos^2\theta + \lambda\cos\theta\sin\theta}{\cos\varphi\sin\theta - \lambda\cos\theta} \\
\frac{\sin\varphi\cos\theta}{\cos\varphi\sin\theta-\lambda\cos\theta}\\
0
\end{pmatrix} \nonumber\\
&= \frac{1}{\cos\varphi\sin\theta - \lambda\cos\theta} \begin{pmatrix}
\cos\varphi \\ \sin\varphi\cos\theta \\ 0
\end{pmatrix}.
\end{align}
In this last step, the plane $\mathcal{P}_\lambda(1) = \text{Span}\left(p_\lambda(1),q_\lambda(1)\right)$ will not depend on $\lambda$ if and only if:
\begin{align}
(C e_r) \cdot \left[p(1) \wedge q(1)\right] = 0.
\end{align}
Here, $p(1) \wedge q(1)$ is colinear to
\begin{align}
\begin{pmatrix} \cos\varphi \\ \sin\varphi\cos\theta \\ 0 \end{pmatrix} \wedge (Cq(0)) &= \begin{pmatrix} \cos\varphi \\ \sin\varphi\cos\theta \\ 0 \end{pmatrix} \wedge \left( \begin{pmatrix} 1 & 0 & 0 \\ 0 & 1 & (1+r)/2 \\ 0 & 0 & -r \end{pmatrix} \begin{pmatrix}
\cos\varphi\cos\theta \\
\sin\varphi \\
-\cos\varphi\sin\theta
\end{pmatrix}\right) \nonumber\\
&= \begin{pmatrix} \cos\varphi \\ \sin\varphi\cos\theta \\ 0 \end{pmatrix} \wedge \begin{pmatrix} \cos\varphi\cos\theta \\ \sin\varphi - \frac{1+r}{2}\cos\varphi\sin\theta \\ r\cos\varphi\sin\theta \end{pmatrix} \nonumber\\
&= \begin{pmatrix}
r \cos\varphi\sin\varphi \cos\theta\sin\theta \\
- r \cos^2\varphi \sin\theta \\
\cos\varphi\sin\varphi - \frac{1+r}{2}\cos^2\varphi\sin\theta - \cos\varphi\sin\varphi \cos^2\theta
\end{pmatrix}.
\end{align}
In this last case $Ce_r$ is equal to:
\begin{align}
C e_r = \begin{pmatrix} 1 & 0 & 0 \\ 0 & 1 & (1+r)/2 \\ 0 & 0 & -r \end{pmatrix} \begin{pmatrix} \sin\theta \\ 0 \\ \cos\theta \end{pmatrix} = \begin{pmatrix} \sin\theta \\ \frac{1+r}{2}\cos\theta \\ -r\cos\theta \end{pmatrix}
\end{align}
so that we find in the end:
\begin{align}
&\hspace{-10mm} (C e_r) \cdot \left[(\cos\varphi\sin\theta)p(1) \wedge q(1)\right] \nonumber\\
&= \begin{pmatrix} \sin\theta \\ \frac{1+r}{2}\cos\theta \\ -r\cos\theta \end{pmatrix} \cdotp \begin{pmatrix}
r \cos\varphi\sin\varphi\cos\theta\sin\theta \\
- r \cos^2\varphi \sin\theta \\
\cos\varphi\sin\varphi - \frac{1+r}{2}\cos^2\varphi\sin\theta - \cos\varphi\sin\varphi \cos^2\theta
\end{pmatrix} \nonumber\\
&= \left(r \cos\varphi \sin\varphi \cos\theta\sin^2\theta \right) \nonumber\\
&\hspace{5mm} + \left(-r\frac{1+r}{2} \cos^2\varphi \cos\theta\sin\theta  \right) \nonumber\\
&\hspace{5mm} + \left( -r \cos\varphi\sin\varphi \cos\theta + r\frac{1+r}{2}\cos^2\varphi \cos\theta\sin\theta + r \cos\varphi\sin\varphi \cos^3\theta\right) \nonumber\\
&= r\cos\varphi\sin\varphi\cos\theta(\sin^2\theta+\cos^2\theta) - r\cos\varphi\sin\varphi\cos\theta = 0.
\end{align}
As a consequence, also in this last case we have shown that any trajectory starting from a plane $\mathcal{P}(0)$ with a collision of type $\mathfrak{b}$, immediately followed by a collision of type $\mathfrak{c}$, is contained in a plane $\mathcal{P}(1)$ after this second collision, and this plane $\mathcal{P}(1)$ depends only on $\mathcal{P}(0)$.\\
Since we treated all the possible cases, the proof of Theorem \ref{THEORSpherical_Reduc} is complete.
\end{proof}

\noindent
Theorem \ref{THEORSpherical_Reduc} has an important consequence. Let us repeat the classical argument of reduction of the dimension.\\ Naively, one can describe the evolution of the particle system with two vectors of $\mathbb{R}^3$, describing respectively the relative positions and velocities at all time. However, it is possible to provide a much more economical description. First one can proceed to the natural reduction, using a discrete dynamical system, where the two relative positions and velocities vectors are given only at the times of collisions. One can reduce the description further using the scalings of the dynamics, by assuming that the initial relative positions and velocities are both normalized. At this step, we have a discrete, $6$-dimensional dynamical system, with initial data lying in a $4$-dimensional space. One can finally reduce the dimensionality further: in the discrete description, a pair of particles is colliding at each step. Therefore, one of the three relative distances is zero at each step. So, up to keep track of which pair is in contact, by using some index $c_k \in \{\mathfrak{a},\mathfrak{b},\mathfrak{c}\}$, the discrete dynamical system can be parametrized by elements in $\left( \{\mathfrak{a},\mathfrak{b},\mathfrak{c}\} \times \mathbb{R}^2 \right) \times \mathbb{R}^3$.\\
\newline
Theorem \ref{THEORSpherical_Reduc} states that a much stronger reduction is possible, proceeding as follows. Let us describe the type of each collision of the particle system by its index $c_k \in \{\mathfrak{a},\mathfrak{b},\mathfrak{c}\}$. This index indicates which particles are in contact at the time of the collision we consider. Then, each collision $k$ can be described by a unique vectorial plane $\mathcal{P}(k)$, as follows: the index $c_k$ indicates in which plane lies the vector of relative positions ($x = 0$ if $c_k = \mathfrak{a}$, $y = 0$ if $c_k = \mathfrak{b}$ and $z=0$ if $c_k = \mathfrak{c}$). Therefore, from $\mathcal{P}(k)$ one can deduce the direction of the vector of positions: it is orientated by the intersection between $\mathcal{P}(k)$ and the plane $x$, $y$ or $z=0$ given by index $c_k$. The normalized position vector $\widetilde{p}(k)$ is then lying in the first octant, at the intersection between the unit sphere, the plane $\mathcal{P}(k)$, and the plane selected by the index $c_k$. Finally, the vector of the relative velocities will be given by the following construction: it will be orientated by the intersection between the plane $\mathcal{P}(k)$ and the tangent plane to the unit sphere at the point $\widetilde{p}(k)$.\\
As a consequence, at each collision, only one integer and two real parameters are required to perform this description: the index $c_k \in \{\mathfrak{a},\mathfrak{b},\mathfrak{c}\}$, and two real angles to describe the orientation of the vectorial plane $\mathcal{P}(k)$.\\
\newline
Of course, this description forgets important information about the original particle system: in particular, we see that the norm of the vector of the relative positions is lost, as well as the normal component (with respect to the tangent plane of the unit sphere at the point $\widetilde{p}(k)$) of the vector of the velocities. In particular, the information on the different times between two consecutive collisions is lost.\\
Nevertheless, the mapping, that can be written as:
\begin{align}
\label{EQUATSpherReducMappi}
\left\{
\begin{array}{rcl}
\{\mathfrak{a},\mathfrak{b},\mathfrak{c}\} \times \mathbb{S}^2 &\rightarrow& \{\mathfrak{a},\mathfrak{b},\mathfrak{c}\} \times \mathbb{S}^2,\\
(c_k,\mathcal{P}(k)) &\mapsto& (c_{k+1},\mathcal{P}(k+1))
\end{array}
\right.
\end{align}
keeps track of the order of the collisions. This information is in particular independent from the initial configurations, provided that it belongs to a fixed plane $\mathcal{P}(0)$.\\
All in all, Theorem \ref{THEORSpherical_Reduc} can be interpreted as the exhibition of a hidden conserved quantity of the $4$-particle system.

\begin{remark}
\label{REMARInfinitely_ManyColli}
If now we make abstraction of the original particle system, and if we consider only the mapping \eqref{EQUATSpherReducMappi}, let us observe that we actually defined a billiard on a portion of the unit sphere $\mathbb{S}^2$. More precisely, we defined a billiard on $\mathcal{O}_{\mathbb{S}^2} = \mathcal{O}\cap \mathbb{S}^2 = \{(x,y,z) \in \mathbb{S}^2\ /\ x \geq 0, y \geq 0, z \geq 0\}$. For this reason, we named this construction the \emph{spherical billiard reduction} of the one-dimensional four-particle system.\\
In particular, for any initial datum, and assuming that the orbit of such a billiard never meets one of the corners of $\mathcal{O}_{\mathbb{S}^2}$, the procedure we described associates to such an initial datum an infinite sequence of collisions.\\
This is a fundamental difference with the physical system of four hard spheres, because some initial configurations might lead to an eventual separation of the particles, that is, trajectories with a finite number of collisions.\\
Actually, the spherical billiard reduction tells that to any initial configuration of the one-dimensional four-particle system, it is possible to associate one single infinite sequence of collisions. In the physical system, it might be that only a finite number of such collisions will indeed be realized, but the sequence of the realized collisions has always to match with the beginning of the infinite sequence provided by the spherical reduction.
\end{remark}

\subsection{The final expressions of the spherical reduction}
\noindent
In the end, as a corollary of Theorem \ref{THEORSpherical_Reduc}, the evolution of the dynamical system can be reduced to a three-dimensional mapping (the third dimension corresponding to the index of the colliding pair). In this section, we give two different expressions for such a mapping.

\subsubsection{Spherical reduction, trigonometric version}
\label{SSectSpherReducTrigo}

\noindent
When writing the normalized position and velocity vectors $p$ and $q$ in coordinates, we obtain the following expressions for the spherical reduction mapping \eqref{EQUATSpherReducMappi}.

\begin{itemize}
\item if \circled{1} and \circled{2} are in contact:
$p(0) = \left(0,\cos\theta,\sin\theta\right)$ and $q(0) = \left(\sin\varphi,-\sin\theta\cos\varphi,\cos\theta\cos\varphi\right)$, and:
\begin{itemize}
    \item if $\varphi \in\ ]0,\pi/2[$, the next collision is of type \circled{2}-\circled{3}, and:\begin{align}
    p(t_1) &= \frac{1}{\sqrt{\cos^2\theta\sin^2\varphi + \cos^2\varphi}}\left(\cos\theta\sin\varphi,0,\cos\varphi\right),\\
    q(t_1) &= \frac{B q(0) - \left(B q(0) \cdot p(t_1)\right) p(t_1)}{\vertii{B q(0) - \left(B q(0) \cdot p(t_1)\right) p(t_1)}},
    \end{align}
    \item if $\varphi \in\ ]\pi/2,\pi[$, the next collision is of type \circled{3}-\circled{4}, and:
    \begin{align}
    p(t_1) &= \frac{1}{\sqrt{\sin^2\theta\sin^2\varphi+\cos^2\varphi}}\left(\sin\theta\sin\varphi,-\cos\varphi,0\right),\\
    q(t_1) &= \frac{C q(0) - \left(C q(0) \cdot p(t_1)\right) p(t_1)}{\vertii{C q(0) - \left(C q(0) \cdot p(t_1)\right) p(t_1)}},
    \end{align}
\end{itemize}
\item if \circled{2} and \circled{3} are in contact:
$p(0) = \left(\sin\theta,0,\cos\theta\right)$ and $q(0) = \left(\cos\theta\cos\varphi,\sin\varphi,-\sin\theta\cos\varphi\right)$, and:
\begin{itemize}
    \item if $\varphi \in\ ]0,\pi/2[$, the next collision is of type \circled{3}-\circled{4}, and:
    \begin{align}
    p(t_1) &= \frac{1}{\sqrt{\cos^2\varphi+\cos^2\theta\sin^2\varphi}}\left(\cos\varphi,\cos\theta\sin\varphi,0\right),\\
    q(t_1) &= \frac{C q(0) - \left(C q(0) \cdot p(t_1)\right) p(t_1)}{\vertii{C q(0) - \left(C q(0) \cdot p(t_1)\right) p(t_1)}},
    \end{align}
    \item if $\varphi \in\ ]\pi/2,\pi[$, the next collision is of type \circled{1}-\circled{2}, and:
    \begin{align}
    p(t_1) &= \frac{1}{\sqrt{\sin^2\theta\sin^2\varphi + \cos^2\varphi}}\left(0,\sin\theta\sin\varphi,-\cos\varphi\right),\\
    q(t_1) &= \frac{A q(0) - \left(A q(0) \cdot p(t_1)\right) p(t_1)}{\vertii{A q(0) - \left(A q(0) \cdot p(t_1)\right) p(t_1)}},
    \end{align}
\end{itemize}
\item if \circled{3} and \circled{4} are in contact:
$p(0) = \left(\cos\theta,\sin\theta,0\right)$ and $q(0) = \left(-\sin\theta\cos\varphi,\cos\theta\cos\varphi,\sin\varphi\right)$, and:
\begin{itemize}
    \item if $\varphi \in\ ]0,\pi/2[$, the next collision is of type \circled{1}-\circled{2}, and:
    \begin{align}
    p(t_1) &= \frac{1}{\sqrt{\cos^2\varphi + \cos^2\theta\sin^2\varphi}}\left(0,\cos\varphi,\cos\theta\sin\varphi\right),\\
    q(t_1) &= \frac{A q(0) - \left(A q(0) \cdot p(t_1)\right) p(t_1)}{\vertii{A q(0) - \left(A q(0) \cdot p(t_1)\right) p(t_1)}},
    \end{align}
    \item if $\varphi \in\ ]\pi/2,\pi[$, the next collision is of type \circled{2}-\circled{3}, and:
    \begin{align}
    p(t_1) &= \frac{1}{\sqrt{\cos^2\varphi + \sin^2\theta\sin^2\varphi}}\left(-\cos\varphi,0,\sin\theta\sin\varphi\right),\\
    q(t_1) &= \frac{B q(0) - \left(B q(0) \cdot p(t_1)\right) p(t_1)}{\vertii{B q(0) - \left(B q(0) \cdot p(t_1)\right) p(t_1)}} \cdotp
    \end{align}
\end{itemize}
\end{itemize}

\subsubsection{Spherical reduction mapping, vectorial version}

\noindent
We can also present the spherical reduction mapping \eqref{EQUATSpherReducMappi} using only cross products. In such a way, each iteration of the mapping is a trilinear mapping.\\
Let us expose the computations that we will use to justify the vectorial version of the spherical reduction mapping, in the particular case when \circled{1} and \circled{2} are initially in contact (the two other cases can be deduced by relabeling the particles and the collisions matrices appropriately).\\
\newline
If the initial position and velocity vectors write:
\begin{align}
p(0) = \begin{pmatrix} 0 \\ \cos\theta \\ \sin\theta \end{pmatrix}, \hspace{5mm} \text{and} \hspace{5mm} q(0) = \begin{pmatrix} \sin\varphi \\ -\sin\theta\cos\varphi \\ \cos\theta\cos\varphi \end{pmatrix},
\end{align}
for $\theta \in [0,\pi/2[$ and $\varphi \in\ ]0,\pi[$, we will \emph{define} the vector $u$, normal to the plane $\mathcal{P}(0) = \text{Span}(p(0),q(0))$ which characterizes the trajectory of the spherical reduction mapping, as:
\begin{align}
u(0) = p(0) \wedge q(0).
\end{align}
Let us emphasize that the order of the two terms in the previous cross product are important to define a self-consistent mapping. In such a way, we have:
\begin{align}
u(0) = \begin{pmatrix} \cos\varphi \\ \sin\theta\sin\varphi \\ -\cos\theta\sin\varphi \end{pmatrix},
\end{align}
so that in particular $e_y \cdot u(0) = \sin\theta\sin\varphi \geq 0$. In general, we \emph{define} the vector $u(k)$, normal to $\mathcal{P}(k)$, such that one has always $e_t \cdot u(k) \geq 0$, where $t = y$ if $k = 1$ (\circled{1} and \circled{2} in contact), $t = z$ if $k = 2$ (\circled{2} and \circled{3} in contact) and $t = x$ if $k = 3$ (\circled{3} and \circled{4} in contact).\\
\newline
Let us now describe how to recover the direction of the vector $u(1)$, normal to the plane $\mathcal{P}(1)$, characterizing the trajectory right after the collision that follows the initial configuration, characterized itself by the plane $\mathcal{P}(0)$.\\
By definition, the vector $u(1)$ will be colinear to, and with the same orientation as the vector:
\begin{align}
p(1) \wedge q(1).
\end{align}
On the one hand, $p(1)$ is deduced from $u(0)$ as follows. If the next collision involves \circled{2} and \circled{3}, then $\varphi \in\ ]0,\pi/2[$. In this case, the plane $\mathcal{P}(0)$ intersects the plane $y=0$ along a line that has the same direction as $p(1)$, and if we consider the vector $u(0) \wedge e_y$, we have:
\begin{align}
u(0) \wedge e_y = \begin{pmatrix} \cos\varphi \\ \sin\theta\sin\varphi \\ -\cos\theta\sin\varphi \end{pmatrix} \wedge \begin{pmatrix} 0 \\ 1 \\ 0 \end{pmatrix} = \begin{pmatrix} \cos\theta\sin\varphi \\ 0 \\ \cos\varphi \end{pmatrix},
\end{align}
so that $p(1)$ and $u(0) \wedge e_y$ are colinear, and have the same orientation, that is $p(1) \cdot ( u(0) \wedge e_y ) = \vertii{p(1)} \cdot \vertii{u(0) \wedge e_y} > 0$  (because all the coordinates that are non-zero of the latter vector are positive). Let us also observe that if the next collision involves instead \circled{3} and \circled{4}, then if we compute $u(0) \wedge e_z$ we find:
\begin{align}
u(0) \wedge e_z = \begin{pmatrix} \cos\varphi \\ \sin\theta\sin\varphi \\ -\cos\theta\sin\varphi \end{pmatrix} \wedge \begin{pmatrix} 0 \\ 0 \\ 1 \end{pmatrix} = \begin{pmatrix} \sin\theta\sin\varphi \\ -\cos\varphi \\ 0 \end{pmatrix},
\end{align}
which is again a vector colinear with $p(1)$ and with all its non-zero coordinates that are positive (because this time $\cos\varphi$ would be negative).\\
It remains to compute $q(1)$. In the case when the collision that follows the initial configuration is of type \circled{2}-\circled{3}, $q(1)$ would be orientated by $B q(0)$. For the vector $u(0)$ being given, and only this vector (that is, without knowing $p(0)$ and $q(0)$), we have to deduce $q(0)$. By definition, $q(0)$ is orientated by the line corresponding to the intersection between $\mathcal{P}(0)$ and the tangent vector to the unit sphere at the point $p(0)$. We have then first to determine $p(0)$ to determine $q(0)$. Let us observe that the vector $u(0) \wedge e_x$ will be colinear to $p(0)$, but since we have:
\begin{align}
u(0) \wedge e_x = \begin{pmatrix} \cos\varphi \\ \sin\theta\sin\varphi \\ -\cos\theta\sin\varphi \end{pmatrix} \wedge \begin{pmatrix} 1 \\ 0 \\ 0 \end{pmatrix} = \begin{pmatrix} 0 \\ -\cos\theta\sin\varphi \\ -\sin\theta\sin\varphi \end{pmatrix},
\end{align}
these two vectors do not have the same orientation. To recover $p(0)$ from $u(0)$, we should therefore consider instead $e_x \wedge u(0)$. From this point, we can also recover $q(0)$. Since we have:
\begin{align}
u(0) \wedge p(0) = \begin{pmatrix} \cos\varphi \\ \sin\theta\sin\varphi \\ -\cos\theta\sin\varphi \end{pmatrix} \wedge \begin{pmatrix} 0 \\ \cos\theta \\ \sin\theta \end{pmatrix} = \begin{pmatrix} \sin\varphi \\ -\sin\theta\cos\varphi \\ \cos\theta\cos\varphi \end{pmatrix} = q(0),
\end{align}
we can conclude. From $u(0)$, we define $u(1)$ as:
\begin{align}
\label{EQUATTriliSpherMappi}
u(1) = \frac{\widetilde{u}(1)}{\vertii{\widetilde{u}(1)}}, \text{   with:  } \widetilde{u}(1) = \underbrace{\left[ u(0) \wedge e_y \right]}_{\simeq p(1)} \wedge \underbrace{\left[ B \left( u(0) \wedge \left[ e_x \wedge u(0) \right] \right) \right]}_{\simeq q(1)\text{ when projected on }p(1)^\perp}
\end{align}
in the case when the next collision is of type \circled{2}-\circled{3} (the notation $\simeq p(1)$ means that we produced a vector colinear with $p(1)$ and with the same orientation, and as:
\begin{align}
u(1) = \frac{\widetilde{u}(1)}{\vertii{\widetilde{u}(1)}}, \text{   with:  } \widetilde{u}(1) = \underbrace{\left[ u(0) \wedge e_y \right]}_{\simeq p(1)} \wedge \underbrace{\left[ C \left( u(0) \wedge \left[ e_x \wedge u(0) \right] \right) \right]}_{\simeq q(1)\text{ when projected on }p(1)^\perp}
\end{align}
when the next collision is of type \circled{3}-\circled{4}.\\
We define $u(1)$ in a similar manner if the initial configuration is such that another pair of particles is in contact.

\begin{remark}
One has to be careful with the trilinear mapping \eqref{EQUATTriliSpherMappi}! The order of the different parentheses is of crucial importance, because the cross product is \emph{not} associative.
\end{remark}

\subsection{Impossible collision sequences for the spherical reduction mapping}

\noindent
Let us conclude the theoretical discussion concerning the spherical reduction mapping with studying the possible collision orders.\\
In the case of the original four-particle system, it is clear that when a collision of type $\mathfrak{a}$ and another collision of type $\mathfrak{c}$ take place consecutively, it is not possible that such a pair is followed immediately by another collision of type $\mathfrak{a}$ or $\mathfrak{c}$, because the pairs of particles \circled{1}-\circled{2} and \circled{3}-\circled{4} are separating.\\
In the case of the spherical reduction mapping, such a result is not clear a priori, because any trajectory presents infinitely many collisions. Indeed, it might be possible that for trajectories of the spherical reduction mapping that are associated to trajectories for the original four-particle system such that the particles eventually separate, such trajectories present collision sequences containing $\mathfrak{aca}$ or $\mathfrak{cac}$.\\
We will see that it is never the case. To this regard, the collision patterns of trajectories for the spherical reduction mapping are not essentially different from the trajectories of the original four-particle system.

\begin{proposition}[Impossibility of the sequences $\mathfrak{aca}$ and $\mathfrak{cac}$ for the spherical reduction mapping]
\label{PROPOImposSeqACA_CAC}
Let us consider a trajectory $\left( \left(c_k,\mathcal{P}(k)\right)\right)_k$ of the spherical reduction mapping \eqref{EQUATSpherReducMappi}.\\
Then, if there exists one collision index $k \in \mathbb{N}^*$ such that $c_k = \mathfrak{a}$ and $c_{k+1} = \mathfrak{c}$, necessarily $c_{k+2} \neq \mathfrak{a}$.\\
Similarly, if there exists one collision index $k \in \mathbb{N}^*$ such that $c_k = \mathfrak{c}$ and $c_{k+1} = \mathfrak{a}$, necessarily $c_{k+2} \neq \mathfrak{c}$.
\end{proposition}

\begin{proof}
We will prove only the first statement, since the proof of the second one follows by symmetry. We will use the trigonometric version of the representation of the spherical reduction mapping \eqref{EQUATSpherReducMappi}.\\
Let us assume that $c_k = \mathfrak{a}$. Right after such a collision, the position and velocity vectors of the system write respectively:
\begin{align}
p(k) = \left(0,\cos\theta,\sin\theta\right) \hspace{3mm} \text{and} \hspace{3mm} q(k) = \left(\sin\varphi,-\sin\theta\cos\varphi,\cos\theta\cos\varphi\right).
\end{align}
By assumption, the collision that follows is of type $\mathfrak{c}$. In particular, we have $\varphi \in\ ]\pi/2,\pi[$. We can also deduce the next position and velocity vectors: $p(k+1)$ is colinear to:
\begin{align}
\left(\sin\theta\sin\varphi,-\cos\varphi,0\right),
\end{align}
while $q(k+1)$ equal to the following vector, up to be multiplied by a positive scalar:
\begin{align}
Cq(k) - \left( Cq(k) \cdot p(k) \right) p(k)
\end{align}
with
\begin{align}
Cq(k) = \begin{pmatrix} 1 & 0 & 0 \\ 0 & 1 & (1+r)/2 \\ 0 & 0 & -r \end{pmatrix} \begin{pmatrix} \sin\varphi \\ -\sin\theta\cos\varphi \\ \cos\theta\cos\varphi \end{pmatrix} = \begin{pmatrix} \sin\varphi \\ -\sin\theta\cos\varphi + \frac{(1+r)}{2}\cos\theta\cos\varphi \\ -r\cos\theta\cos\varphi \end{pmatrix},
\end{align}
so that
\begin{align}
Cq(k) - \big( Cq(k) \cdot& p(k) \big) p(k)\nonumber\\
&\hspace{-10mm}= \begin{pmatrix} \sin\varphi \\ -\sin\theta\cos\varphi + \frac{(1+r)}{2}\cos\theta\cos\varphi \\ -r\cos\theta\cos\varphi \end{pmatrix} \nonumber\\
&\hspace{5mm} - \left( -\cos\theta\sin\theta\cos\varphi + \frac{(1+r)}{2}\cos^2\theta\cos\varphi - r\cos\theta\sin\theta\cos\varphi \right) \begin{pmatrix} 0 \\ \cos\theta \\ \sin\theta \end{pmatrix} \nonumber\\
&\hspace{-10mm}= \begin{pmatrix}
\sin\varphi \\
-\sin\theta\cos\varphi + \frac{(1+r)}{2}\cos\theta\cos\varphi + (1+r) \cos^2\theta\sin\theta \cos\varphi - \frac{(1+r)}{2}\cos^3\theta\cos\varphi \\
-r\cos\theta\cos\varphi + (1+r) \cos\theta\sin^2\theta\cos\varphi - \frac{(1+r)}{2} \cos^2\theta\sin\theta\cos\varphi
\end{pmatrix}.
\end{align}
Keeping in mind that $q(k+1)$ is a vector of the form $(-\sin\theta'\cos\varphi',\cos\theta'\cos\varphi',\sin\varphi')$, we can determine the sign of $\cos\varphi'$. Indeed, $\sin\varphi$, which is positive (because by definition $\varphi \in\ ]0,\pi[$), has the same sign as $-\sin\theta'\cos\varphi'$. Therefore, since $\theta \in\ ]0,\pi/2[$, we deduce that $\cos\varphi' < 0$.\\
In other words, $\varphi' \in\ ]\pi/2,\pi[$, so that the $(k+2)$-th collision is necessarily of type $\mathfrak{b}$. This concludes the proof of Proposition \ref{PROPOImposSeqACA_CAC}.
\end{proof}

\begin{corollary}
\label{COROLDescrColliPatte}
For any trajectory $\left( \left(c_k,\mathcal{P}(k)\right)\right)_k$ of the spherical reduction mapping \eqref{EQUATSpherReducMappi}, the number of collisions of type $\mathfrak{b}$ is infinite, that is we have:
\begin{align}
\left\vert \left\{ k \in \mathbb{N}^* / \/c_k = \mathfrak{b} \right\} \right\vert = +\infty,
\end{align}
and no more than two collisions that are not of type $\mathfrak{b}$ can take place between two consecutive collisions of type $\mathfrak{b}$, that is:
\begin{align}
\text{if} \hspace{3mm} c_k = \mathfrak{b},\hspace{3mm} \text{then} \hspace{3mm} c_{k+2} = \mathfrak{b} \hspace{3mm} \text{or} \hspace{3mm} c_{k+3} = \mathfrak{b}.
\end{align}
\end{corollary}

\begin{remark}
As a consequence of Corollary \ref{COROLDescrColliPatte}, we deduce that the collision pattern $c_1 c_2 \dots$ of any trajectory $\left( \left(c_k,\mathcal{P}(k)\right)\right)_k$ of the spherical reduction is obtained by the concatenation of only four elementary ``blocks'' of collisions, namely $\mathfrak{ab}$, $\mathfrak{cb}$, $\mathfrak{acb}$ and $\mathfrak{cab}$. 
\end{remark}

\section{Numerical simulations of the spherical reduction}

\noindent
In this section, we will describe numerical simulations of the spherical reduction mapping.\\
The spherical reduction \eqref{EQUATSpherReducMappi} enables to represent the configurations of the particle system during a collision with the help of a single three-dimensional unit vector (the normal vector to the tangent plane $\mathcal{P}(k)$), and an integer between $1$ and $3$ (to describe which pair of particles is involved in the collision). In particular, in Section \ref{SSectSpherReducTrigo}, we saw that it is possible to represent the position and velocity vectors with, respectively, a single real variable, corresponding to the angles $\theta$ and $\varphi$.\\
As a consequence, this representation enables graphical visualizations of the trajectories, contrary to the situation in which one represents the positions and velocity vectors as general elements of $\mathbb{R}^3$.

\subsection{General approach for the numerical simulations}

\noindent
We will use the trigonometric representation of the spherical reduction mapping \eqref{EQUATSpherReducMappi}, described in Section \ref{SSectSpherReducTrigo}. In order to compare similar configurations of the particle system, we will represent only collisions of type $\mathfrak{b}$, when the pair \circled{2}-\circled{3} is in contact. Such a choice is natural, since when an inelastic collapse of four particles takes place, the collisions of type $\mathfrak{b}$ necessarily take place an infinite number of times. This is of course true for the original system of four particles, and we have proven it is also the case for the spherical reduction mapping (see Corollary \ref{COROLDescrColliPatte}).\\
All the trajectories of the spherical reduction mapping represented below were computed until the $10^4$-th collision. Let us observe once again that this is possible to compute an arbitrary number of collisions for any trajectory, see Remark \ref{REMARInfinitely_ManyColli} (contrary to the original particle system, since the physical particles can end up colliding after a certain time, and disperse, evolving eventually according to the free transport).\\
For the trajectories that the program can compute successfully until the $10^4$-th collision, we store and represent only the last $5\times 10^2$ collisions of type $\mathfrak{b}$, in order to focus on the long-time behaviour of the trajectories. More precisely, our objective is to represent the $\omega$-limit sets of the trajectories of \eqref{EQUATSpherReducMappi}.\\
Note that, even though in theory the number of collisions for any trajectory of the spherical reduction mapping is infinite, we introduced in the program a break function, stopping the computation of a trajectory when some numerical singularity is reached. Namely, it might be that the difference between two collision times becomes smaller than the computer precision. Another type of numerical singularity corresponds to the impossibility to determine the next collision (the angle $\varphi$ of the representation of Section \ref{SSectSpherReducTrigo} may be such that $\vert \cos\varphi \vert$ is smaller than the computer precision, preventing to evaluate safely the sign of this quantity). We excluded these cases using break functions.\\
The simulations presented below are representations of the trajectories in the $(r,\varphi)$ plane. In summary, for different choices of the restitution coefficient $r$, represented on the horizontal axis, we plotted clouds of points $(r,\varphi_n)_{1 \leq n \leq 500}$, corresponding to the last $500$ collisions of type $\mathfrak{b}$ of trajectories, that is, we plotted on the vertical axis the angles $\varphi_n$, corresponding to the velocity vectors $q(\mathfrak{b}_n) = \left(\cos\theta_n\cos\varphi_n,\sin\varphi_n,-\sin\theta_n\cos\varphi_n\right)$.\\
Note that representing the angle $\theta$ instead of $\varphi$ provides very similar results. For this reason, we present here only projections of the trajectories in the $(r,\varphi)$ plane.\\
Finally, the main difference between the following simulations is the number of trajectories computed for a fixed restitution coefficient $r$: in general we represent $8$ different trajectories, but we will focus also in some cases on one single trajectory.

\subsection{First simulation: representation for all possible restitution coefficients $r$}

\noindent
Since the spherical reduction mapping enables to consider arbitrarily large number of collisions, independently from the restitution coefficient $r$, we can represent trajectories corresponding to the ``physical'' ranges of $r$ already studied in \cite{CDKK999} and \cite{BeCa999}, but we can also reach collapses taking place in more general ranges, that have no chance to be caught considering the original particle system (as a consequence of Theorem 1 page 460 in \cite{BeCa999}). Figure \ref{fig:Simul_1} is a projection on the $(r,\varphi)$ plane of trajectories of the spherical reduction mapping \eqref{EQUATSpherReducMappi}, for $0.072 \leq r \leq 0.999$.\\

\begin{figure}
\centering
    \includegraphics[trim = 0cm 1.5cm 0cm 2cm, width=1\linewidth]{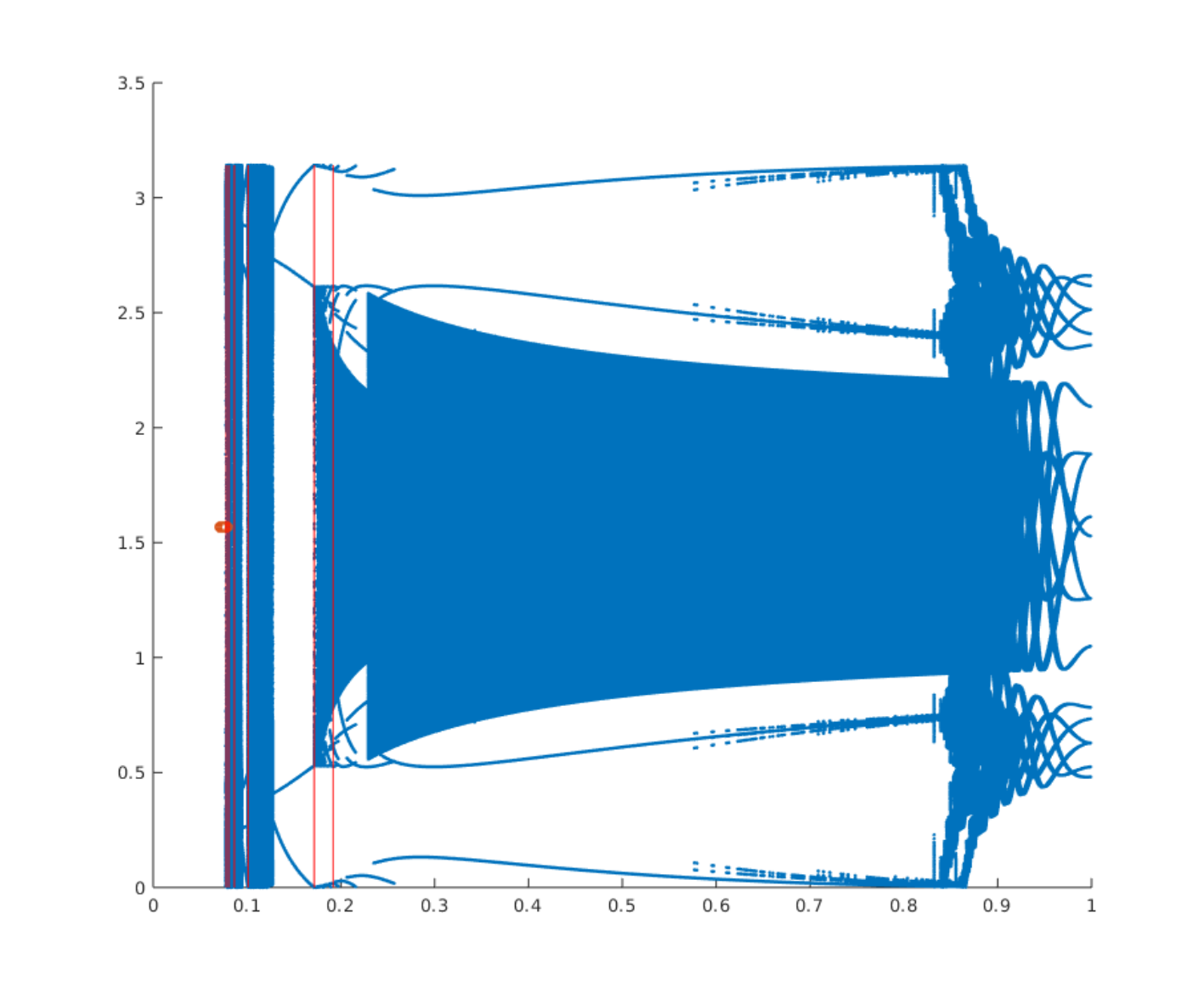}
    \caption{Representation in the $(r,\varphi)$ plane of the last $500$ collisions of type $\mathfrak{b}$ of $8$ trajectories presenting at least $10^4$ collisions, plotted for each restitution coefficient $0.072 \leq r \leq 0.999$. The step $\Delta_r$ between two restitution coefficients is $\Delta_r = 2.5\times 10^{-4}$ ($3709$ different values of $r$). For each restitution coefficient $r$, the initial configurations $(\theta_0,\varphi_0)$ of the trajectories are chosen on a regular grid, with $2$ values of $\theta_0$, and $4$ values of $\varphi_0$.}
    \label{fig:Simul_1}
\end{figure}

\noindent
Let us observe on Figure \ref{fig:Simul_1} that, on the one hand, for most of the values of the restitution coefficient $r$ (for approximately $r \leq 0.9$), the orbits of the trajectories seems to fill in a dense manner intervals of the vertical segment $[0,\pi]$ in which lie the angles $\varphi_n$. This could suggest a chaotic behaviour of the dynamical system \eqref{EQUATSpherReducMappi}.\\
On the other hand, for $r$ small enough ($r \leq 0.17$), there exist intervals of value of the restitution coefficient $r$ for which the orbits of the dynamical system accumulate around a finite number of points. This corresponds to the periodic patterns already observed and described in \cite{CDKK999}. We will discuss this point further using Figure \ref{fig:Simul_2} and its magnifications in the following section.\\
Therefore, depending on the restitution coefficient $r$, we observe windows of stability, separated by windows of apparent chaos (as already observed in \cite{CDKK999}). Let us note that this behaviour has strong similarity with what is known about the logistic map, and its bifurcation diagram (see for instance \cite{Draz005}).
\newline
\newline
The vertical red lines represented correspond to the critical values of $r$ determined in \cite{CDKK999}. On the left hand side, a short horizontal segment appears around the point of coordinates $(0.08,\pi/2)$. It is the collection of the points plotted when the program was unable to compute successfully the trajectories until their $10^4$-th collision. For a given restitution coefficient $r$, each time the break function of the program was activated and a trajectory was not computed until the end, a red dot was plotted at $(r,\pi/2)$.\\
These activations of the break function are accumulated close to the region $r \simeq 0.072$. We will describe in the next section a possible interpretation of the onset of these numerical singularities, linked with the windows of stability.\\
\newline
In the opposite case, for $r$ close to $1$, the trajectories seem to accumulate around a finite number of values of $\varphi$. As $r$ changes, these values describe curves, that oscillate faster and faster as $r$ is decreasing.\\
In addition, the $\omega$-limit sets for $r$ very close to $1$ seem to be discrete sets. But as $r$ is decreasing, the orbits seem to cover small intervals instead of points, in such a way that the $\omega$-limit sets seem to grow, and do not seem to be discrete anymore, as it can be seen for $r \simeq 0.92$.\\
Therefore, this behaviour suggests that the orbits can be quasi-periodic. In addition, in the regions for which the values of $\varphi$ seem to cover a large portion of the interval $[0,\pi]$, it is actually possible that each orbit does not cover the whole interval. On the one hand, it is possible that each orbit covers only a finite collection of sub-intervals. In this way, the $\omega$-limit set of a single trajectory might be not dense in general, although not discrete in general neither. On the other hand, the union of the $\omega$-limit sets, for all the orbits, can be dense, in agreement with what we observe in the windows of apparent chaos.

\subsection{Second simulation: representation for $r$ in the region where collapse takes place for the original particle system}

\noindent
For this second simulation, we focus on the ``physical'' cases. More precisely, it is known (\cite{BeCa999} and \cite{CDKK999}) that inelastic collapse of four particles takes place for a set of initial data of positive measure, for the original system, for $r$ smaller than $r_\text{stabi.} = 3-2\sqrt{2} \simeq 0.17157$. In addition, in \cite{CDKK999} the authors constructed unstable self-similar collapses for $r$ until $r_\text{exist.} \simeq 0.19166$, where $r_\text{exist.}$ is the only real root in $[0,1]$ of the polynomial:
\begin{align}
P(r) = 17r^6 - 138r^5 + 831r^4 - 3148r^3 + 831r^2 - 138r + 17.
\end{align}
As for the lower bound, we will stay away from the region for which the three-particle collapse can take place. Such a collapse can happen only if $r \leq r_\text{crit.} = 7-4\sqrt{3} \simeq 0.07180$.\\
The second simulation is then performed for $0.075 \leq r \leq 0.192$, and is illustrated in Figure \ref{fig:Simul_2}.

\begin{figure}
\centering
    \includegraphics[trim = 0cm 1.5cm 0cm 0cm, width=1\linewidth]{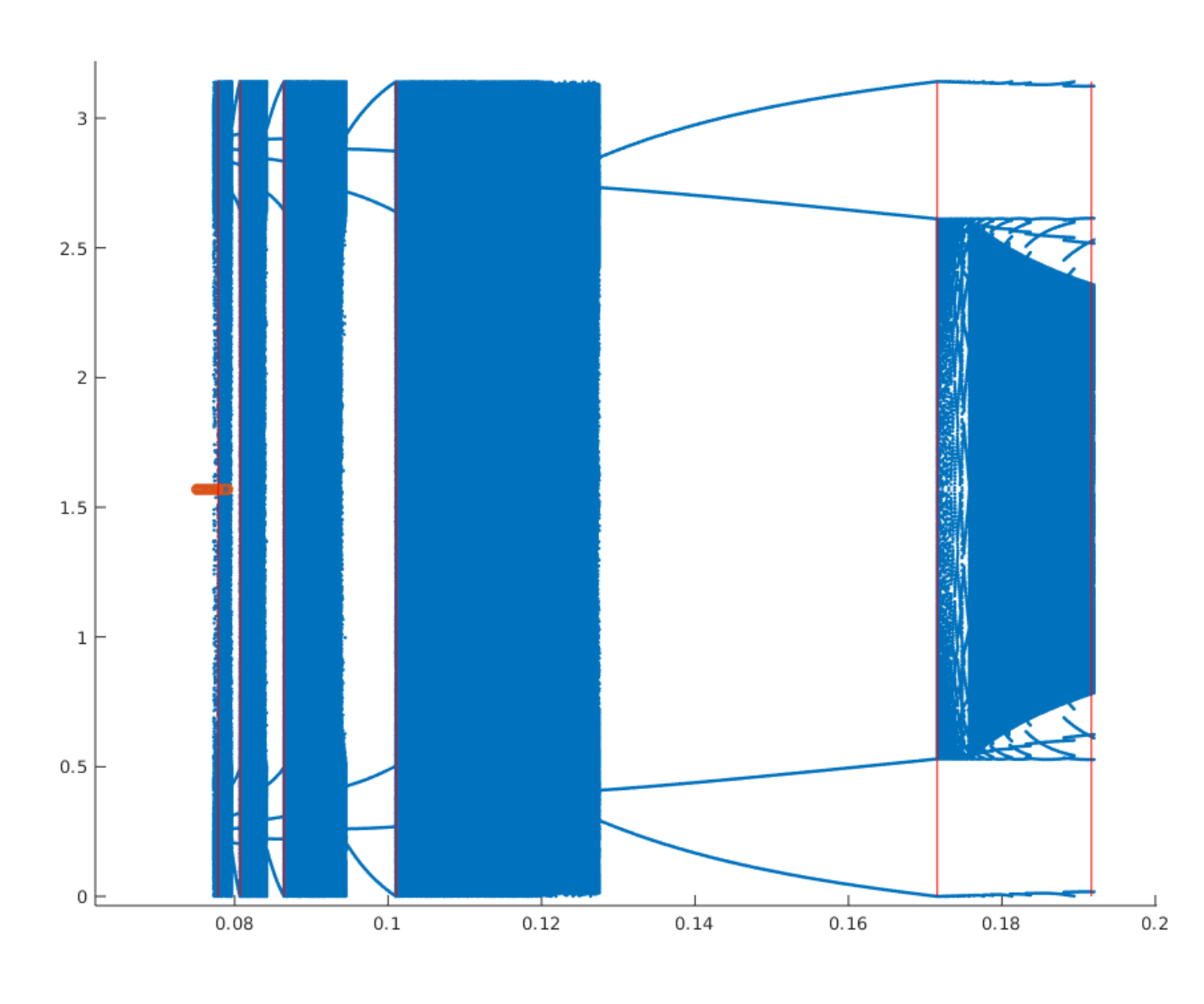}
    \caption{Representation in the $(r,\varphi)$ plane of the last $500$ collisions of type $\mathfrak{b}$ of $8$ trajectories presenting at least $10^4$ collisions, plotted for each restitution coefficient $0.075 \leq r \leq 0.192$. The step $\Delta_r$ between two restitution coefficients is $\Delta_r = 2.5\times 10^{-5}$ ($4681$ different values of $r$). For each restitution coefficient $r$, the initial configurations $(\theta_0,\varphi_0)$ of the trajectories are chosen on a regular grid, with $2$ values of $\theta_0$, and $4$ values of $\varphi_0$.}
    \label{fig:Simul_2}
\end{figure}
\noindent
We observe clearly the different windows of stability of Cipra et al. (\cite{CDKK999}). The upper values of $r$ that bound these windows are represented by the vertical red lines on Figure \ref{fig:Simul_2}. From the right to the left, these lines correspond respectively to the restitution coefficients:
\begin{align}
r_{\text{max},(\mathfrak{ab})^2(\mathfrak{cb})^2} &= 3-2\sqrt{2} \simeq 0.17157,\\
r_{\text{max},(\mathfrak{ab})^3(\mathfrak{cb})^3} &= 5-2\sqrt{6} \simeq 0.10102,\\
r_{\text{max},(\mathfrak{ab})^4(\mathfrak{cb})^4} &= 3 + 2\sqrt{2} -2\sqrt{4 + 3\sqrt{2}} \simeq 0.08643,\\
r_{\text{max},(\mathfrak{ab})^5(\mathfrak{cb})^5} &= 4 + \sqrt{5} -2\sqrt{5 + 2\sqrt{5}} \simeq 0.08070,\\
r_{\text{max},(\mathfrak{ab})^6(\mathfrak{cb})^6} &= 3 + 2\sqrt{3} - 2\sqrt{5 + 3\sqrt{3}} \simeq 0.07782.
\end{align}
The reader may consult in particular the Table 1 page 198 in \cite{CDKK999}, and Section 7 ``General $4n$ bounce pattern'' in the same reference for more details on these values.\\
\newline
The largest stability window, between $r \simeq 0.12754$ and $r \simeq 0.17157$, corresponds to the regime for which the self-similar collapse associated to the pattern $(\mathfrak{ab})^2(\mathfrak{cb})^2$ is stable. For $r$ fixed in such a window, the orbits accumulate all around $4$ values of $\varphi$, which correspond to the four different collisions of type $\mathfrak{b}$ in the pattern $(\mathfrak{ab})^2(\mathfrak{cb})^2$.\\
Let us observe that there is no other point around which the orbits accumulate, suggesting that there is no other stable periodic pattern that is stable in such a window. To investigate more precisely this question, it would be necessary to repeat the simulation, with much more than $8$ initial configurations of the particle system.\\
In the other windows, we observe a similar behaviour. The reader may in particular refer to Figure \ref{fig:Simul_2_zoom_1}, which is a magnification of the simulation, focused on the leftmost windows of stability. For the second window ($0.09452 \leq r \leq 0.10102$), the orbits accumulates around $6$ points, corresponding to the $6$ different collisions of type $\mathfrak{b}$ of the pattern $(\mathfrak{ab})^3(\mathfrak{cb})^3$. For the third and fourth windows, we observe respectively $8$ and $10$ points.\\
\newline
For $r$ below $0.078$, the program is unable to detect the patterns $(\mathfrak{ab})^n(\mathfrak{cb})^n$ for $n \geq 6$. However, the mathematical construction of Cipra et al. should be feasible for any value of $n$ (although more and more intricate, technically speaking), and so we believe that the stability windows exist for any $n \geq 2$, accumulating in the region above $r = 7-4\sqrt{3} \simeq 0.07180$.\\
If this scenario is true, we see that for large $n$, the $2n$ first collisions of type $\mathfrak{a}$ and $\mathfrak{b}$ induce a strong dissipation of the relative velocities between the particles \circled{1} and \circled{2}, and \circled{2} and \circled{3}, while the last relative velocity between \circled{3} and \circled{4} remains untouched. This strong unbalance between the relative velocities after half a period would imply that the vector $q$ of the relative velocities of the trajectories converges towards the dominant eigenvector of the matrix $J(BA)^n$ (using the notations of \cite{CDKK999}). Such an eigenvector, when normalized, has at least one extremely small component. If this component is smaller in absolute value than the computer precision, our program would not be able to compute for a large number of collisions such a trajectory.\\
Let us finally observe that the critical value $r_\text{crit.} = 7 - 4\sqrt{3}$ corresponds to the largest restitution coefficient for which the three-particle collapse takes place. This value was determined in \cite{McYo991} and \cite{CoGM995}. In this case, the collapse of three particles corresponds to the infinite repetition of the pattern $\mathfrak{ab}$, so that the patterns $(\mathfrak{ab})^n(\mathfrak{cb})^n$, for $n$ large, can be seen as a perturbation of the three-particle collapse.

\begin{figure}
\centering
\begin{subfigure}{0.45\textwidth}
    \includegraphics[trim = 0.5cm 1.5cm 0.5cm 0cm, width=\linewidth]{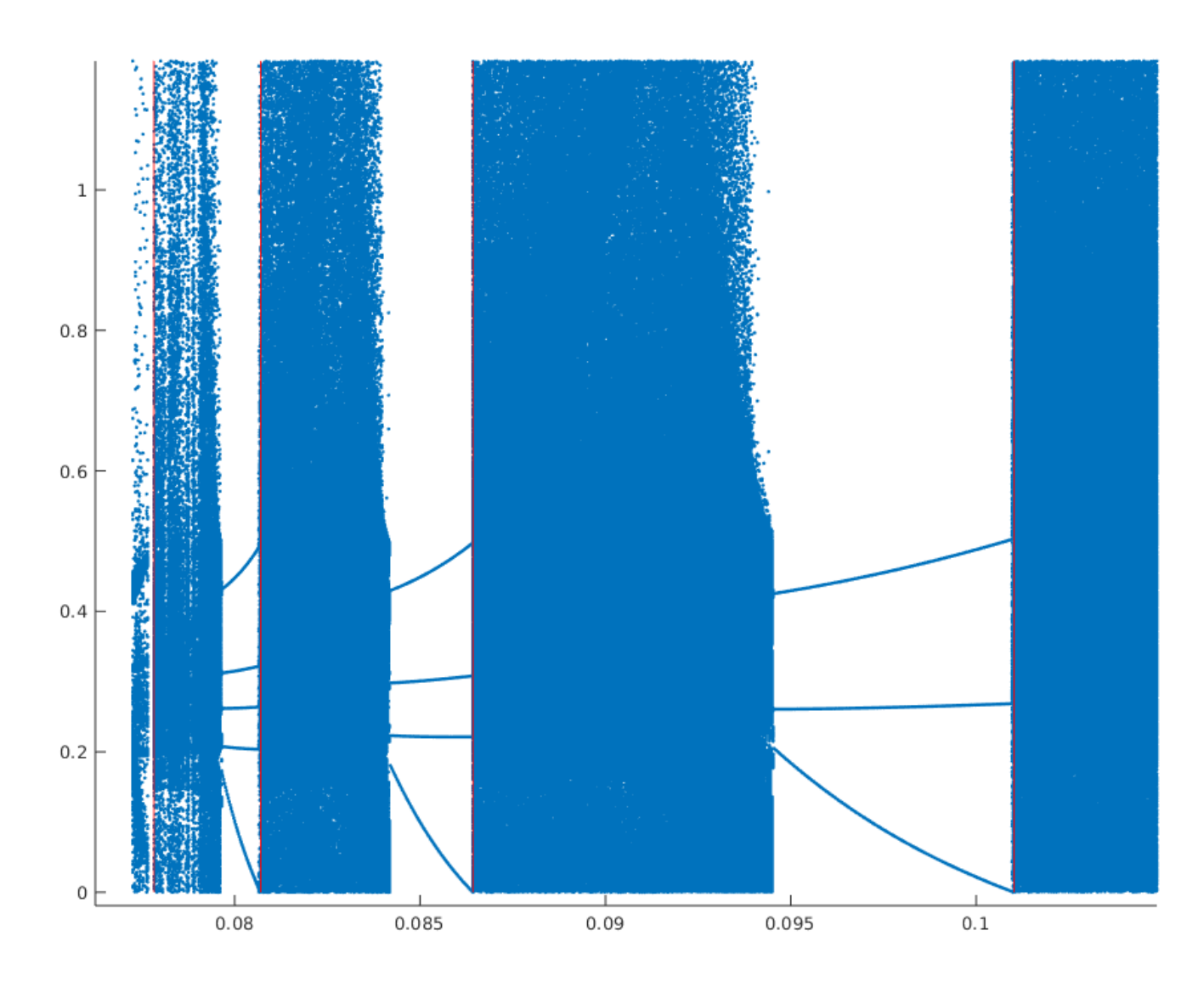}
    \caption{Magnification of Figure \ref{fig:Simul_2}, for $0.075 \leq r \leq 0.105$.}
    \label{fig:Simul_2_zoom_1}
\end{subfigure}
\hfill
\begin{subfigure}{0.45\textwidth}
    \includegraphics[trim = 0.5cm 1.5cm 0.5cm 0cm, width=\linewidth]{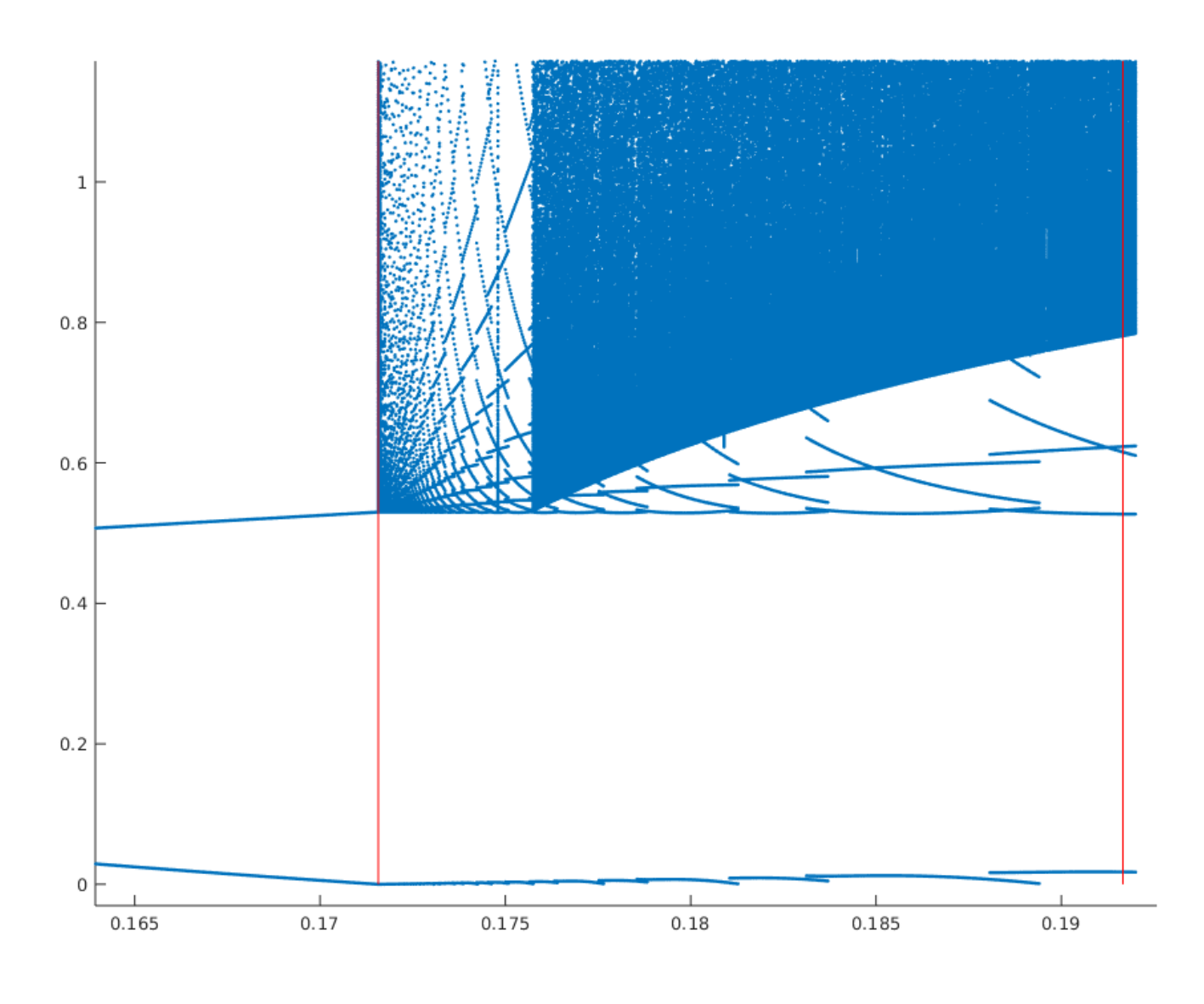}
    \caption{Magnification of Figure \ref{fig:Simul_2}, for $0.165 \leq r \leq 0.192$.}
    \label{fig:Simul_2_zoom_2}
\end{subfigure}
        
\caption{Magnifications of Figure \ref{fig:Simul_2}, focused on the extreme values of the restitution coefficient $r$.}
\label{fig:Zooms_Simul_2}
\end{figure}
\noindent
\newline
Let us now turn to Figure \ref{fig:Simul_2_zoom_2}. We recognize on the left hand side the rightmost part of the largest window of stability, corresponding to the stability region of the pattern $(\mathfrak{ab})^2(\mathfrak{cb})^2$. In particular, the vertical red line at $r = 3-2\sqrt{2} \simeq 0.17157$ corresponds to the maximal value of the restitution coefficient for which there exists an open neighbourhood of initial data around the self-similar configuration that collides according to the periodic pattern $(\mathfrak{ab})^2(\mathfrak{cb})^2$. The second vertical red line corresponds to the threshold restitution coefficient, above which no pattern (periodic or not) of collisions are known for the original four-particle system.\\
It is interesting to note that around $r = 3-2\sqrt{2} \simeq 0.17157$, there is not a sudden transition between the window of stability, and the region of apparent chaos, in which the orbits seem to fill dense parts of the interval $[0,\pi]$. Indeed, it seems that there is a transitory region (for $0.172 \leq r \leq 0.176$ approximately), where the orbits accumulate on finite number of points for $r$ fixed.\\
Unfortunately, and even though the simulation seems to indicate that new periodic patterns might exist in this transitory region, we were not able to detect any.\\
Another simulation focused on this transitory region is represented in Figure \ref{fig:Simul_3}.\\
\begin{figure}[h!]
\centering
    \includegraphics[trim = 0cm 0.5cm 0cm 0cm, width=1\linewidth]{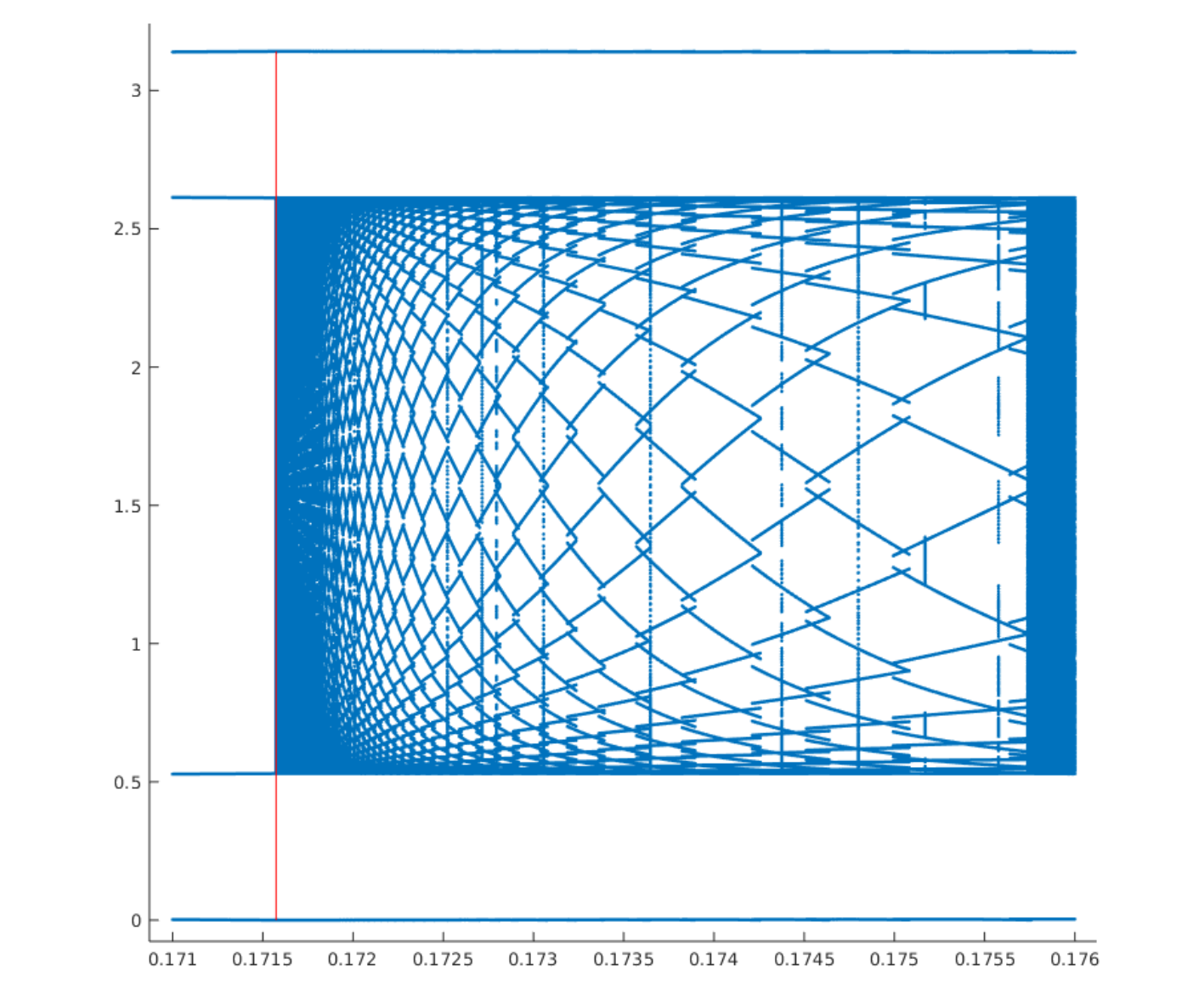}
    \caption{Representation in the $(r,\varphi)$ plane of the transitory region observed in Figure \ref{fig:Simul_2_zoom_2}, which corresponds to restitution coefficients $0.171 \leq r \leq 0.176$. As usual, the last $500$ collisions of type $\mathfrak{b}$ of $8$ trajectories presenting at least $10^4$ collisions are plotted. The step $\Delta_r$ between two restitution coefficients is $\Delta_r = 10^{-6}$ ($5001$ different values of $r$). For each restitution coefficient $r$, the initial configurations $(\theta_0,\varphi_0)$ of the trajectories are chosen on a regular grid, with $2$ values of $\theta_0$, and $4$ values of $\varphi_0$.}
    \label{fig:Simul_3}
\end{figure}
\begin{figure}[h!]
\centering
    \includegraphics[trim = 0cm 1.5cm 0cm 2cm, width=1\linewidth]{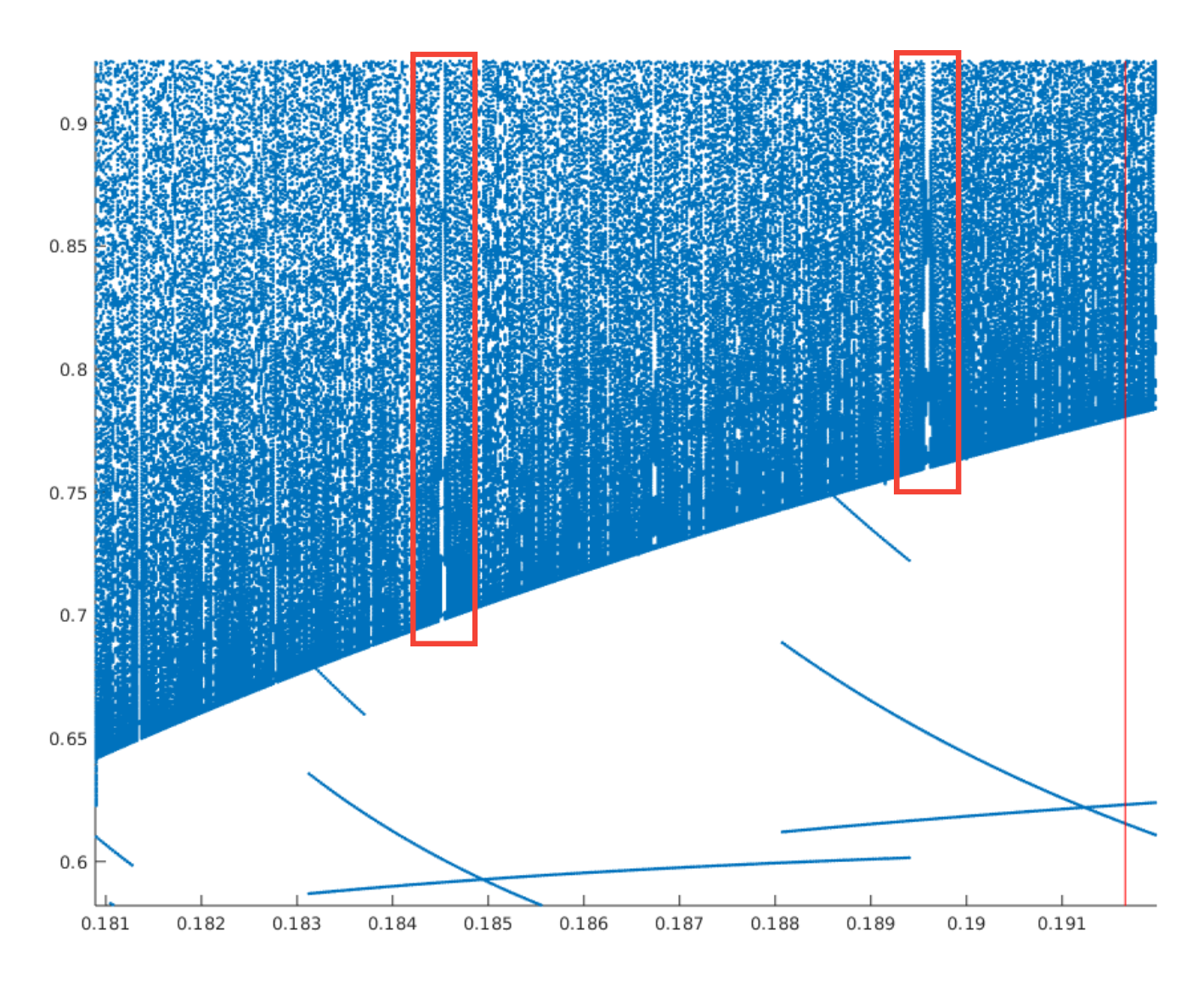}
    \caption{Magnification of Figure \ref{fig:Simul_2}, focused on the region $0.181 \leq r \leq 0.192$.}
    \label{fig:Simul_2_zoom_3}
\end{figure}
\noindent
\newline
Finally, let us turn to one last magnification of Figure \ref{fig:Simul_2}, pictured in Figure \ref{fig:Simul_2_zoom_3}. This magnification is performed in the stripe $0.181 \leq r \leq 0.192$, which corresponds to a region of apparent chaos.\\
It is clear that for most of the restitution coefficients, the orbits seem to cover a dense interval of the segment $[0,\pi]$ (although the whole segment is never filled completely). Nevertheless, and it is quite surprising, we can discern, among this densely filled region, very thin vertical stripes, in which the orbits seem to accumulate on isolated points. We can see two thin stripes in Figure \ref{fig:Simul_2_zoom_3}, highlighted by two red rectangles: a first one for $r \simeq 0.1845$, and another one for $r \simeq 0.1896$.\\
We will dedicate the last section to these thin stripes.

\subsection{Third simulation: the thin stripes of stability}

\noindent
Motivated by the observation of thin stripes in Figure \ref{fig:Simul_2_zoom_3}, we present our final simulations, corresponding to restitution coefficients around these thin stripes.\\
\begin{figure}[h!]
\label{fig:Simul_4_5_6}
\centering
\begin{subfigure}{0.3\textwidth}
    \includegraphics[trim = 2cm -1.2cm 2cm 0cm, width=\linewidth]{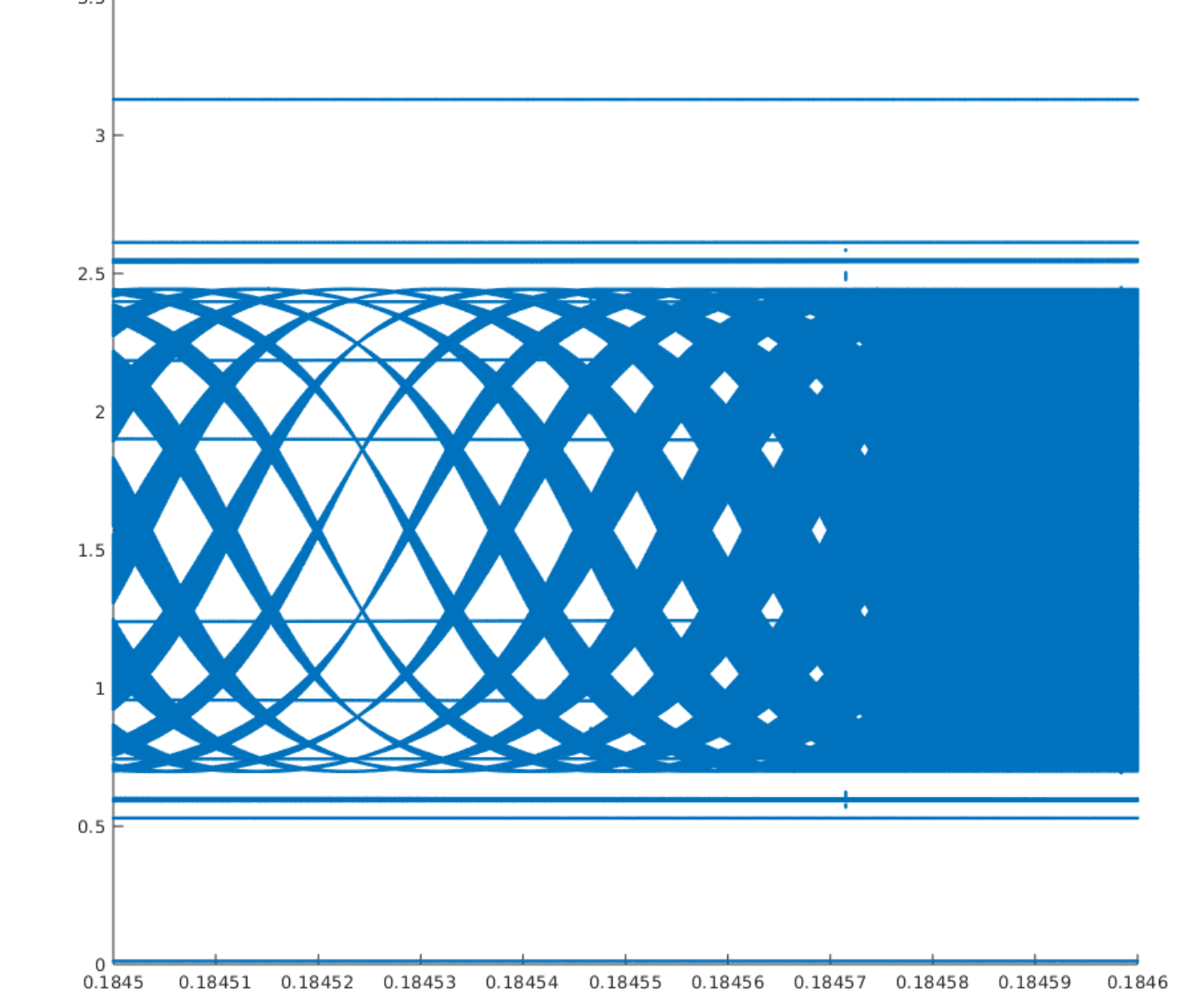}
    \caption{Representation of $8$ trajectories in the first thin stripe in the $(r,\varphi)$ plane.\newline
 The last $500$ collisions of type $\mathfrak{b}$ are plotted for each trajectory and per restitution coefficient in the first thin stripe. $0.1845 \leq r \leq 0.1846$, $\Delta_r = 10^{-7}$ ($1001$ different values of $r$).}
    \label{fig:Simul_4}
\end{subfigure}
\hfill
\begin{subfigure}{0.3\textwidth}
    \includegraphics[trim = 2cm 0.5cm 2cm 0cm, width=\linewidth]{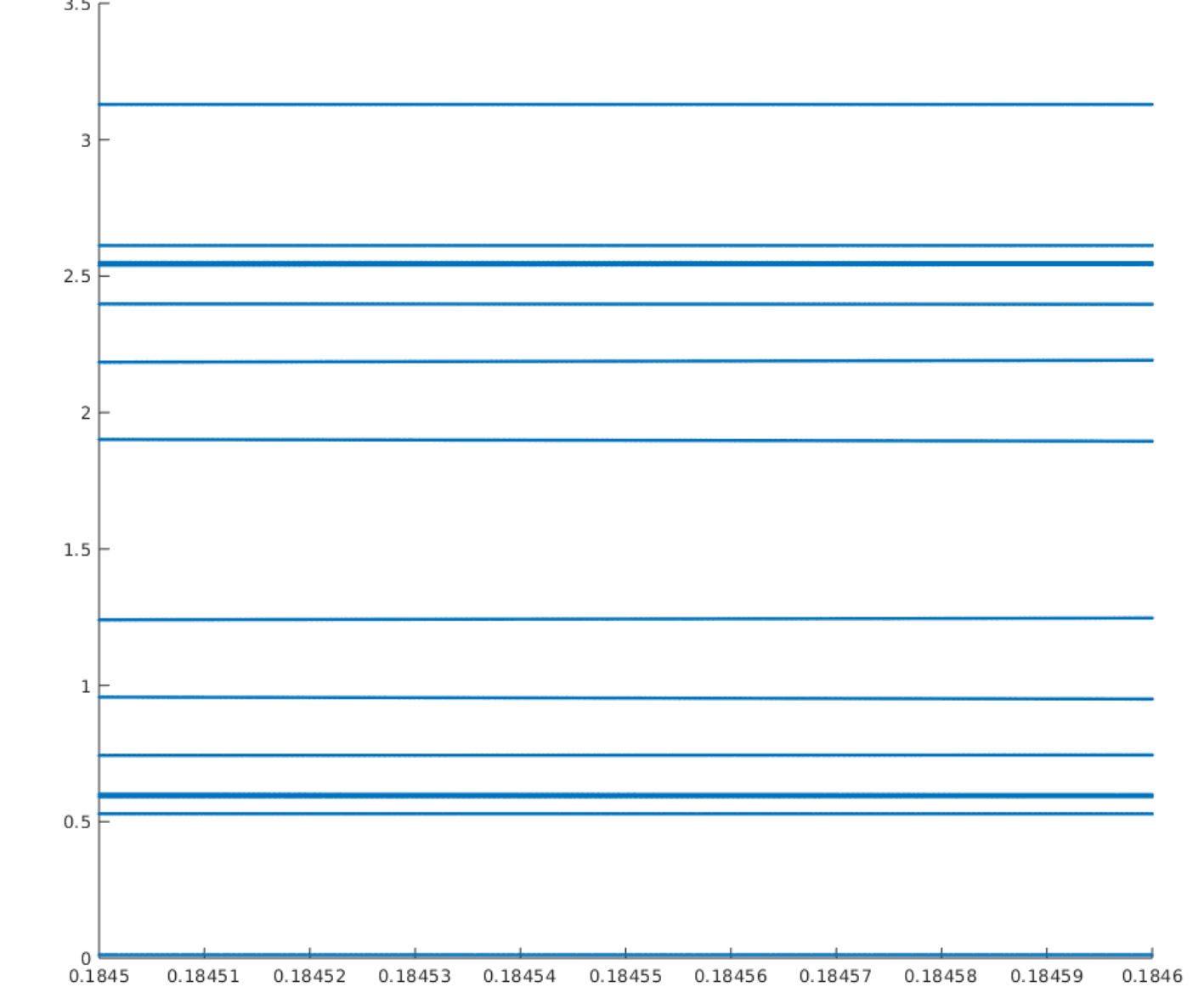}
    \caption{Representation in the $(r,\varphi)$ plane of a single trajectory per restitution coefficient in the first thin stripe.\newline
The last $500$ collisions of type $\mathfrak{b}$ are plotted for $0.1845 \leq r \leq 0.1846$, $\Delta_r = 10^{-7}$ ($1001$ different values of $r$). The initial configuration is chosen randomly.}
    \label{fig:Simul_5}
\end{subfigure}
\hfill
\begin{subfigure}{0.3\textwidth}
    \includegraphics[trim = 2cm 0.5cm 2cm -5cm, width=\linewidth]{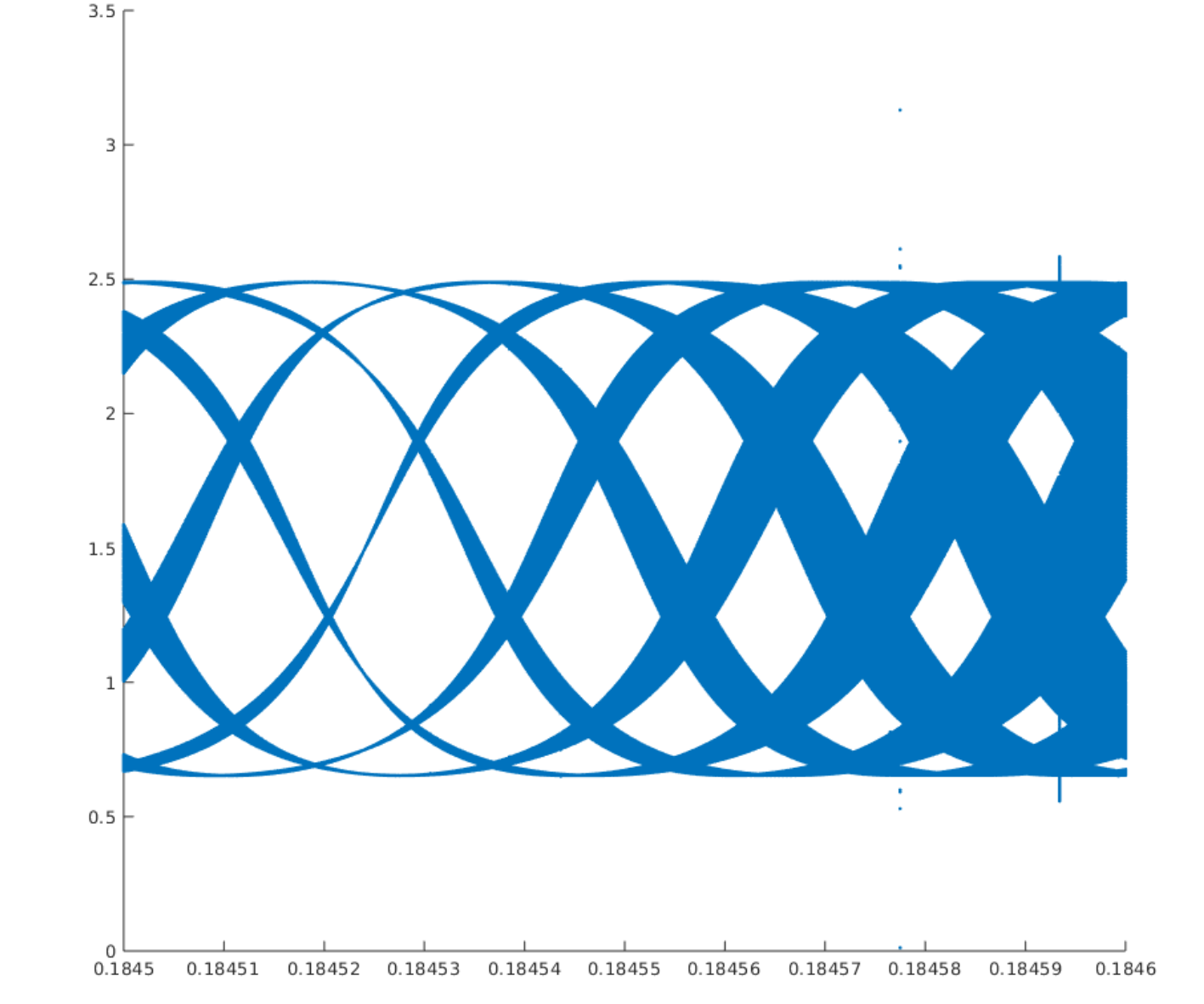}
    \caption{Representation in the $(r,\varphi)$ plane of another single trajectory per restitution coefficient in the first thin stripe.\newline
The last $500$ collisions of type $\mathfrak{b}$ are plotted for $0.1845 \leq r \leq 0.1846$, $\Delta_r = 10^{-7}$ ($1001$ different values of $r$). The initial configuration is chosen randomly.}
    \label{fig:Simul_6}
\end{subfigure}
\caption{}
\end{figure}
\newline
The first simulation concerning the thin stripes is pictured in Figure \ref{fig:Simul_4}, corresponding to $0.1845 \leq r \leq 0.1846$. On this figure, we observe clearly that the orbits do not accumulate anymore on large parts of the segment $[0,\pi]$ for $r \simeq 0.184525$. It seems that there is a particular value of $r$ for which all the trajectories accumulate on discrete sets.\\
When the value of $r$ changes slightly, the $\omega$-limit sets are not discrete anymore, and these $\omega$-limit sets cover larger and larger parts of the interval $[0,\pi]$.\\
\newline
On Figures \ref{fig:Simul_5} and \ref{fig:Simul_6} are represented the respective $\omega$-limit sets associated to two distinct trajectories. For each figure, one initial configuration is chosen randomly. Once this configuration is chosen, the trajectories starting from this fixed configuration are computed for all the different values of $r$.\\
We observe that when one considers trajectories associated to different initial data, not only the $\omega$-limit sets are different, but they are also of a different nature. Indeed, on Figure \ref{fig:Simul_5} the trajectory accumulates around an $\omega$-limit set that seems to be discrete, and that does not change a lot when $r$ varies. On Figure \ref{fig:Simul_6}, another initial configuration gives rise to a completely different $\omega$-limit set, which is not discrete in general (except for very special values of the restitution coefficient), and that changes brutally when $r$ varies.\\
In the end, Figure \ref{fig:Simul_4} is the superposition of the different $\omega$-limit sets that are obtained. It is remarkable that the $\omega$-limit sets seem to become discrete for the same values of $r$, regardless the initial configurations. This phenomenon explains why we were able to detect thin stripes in the region of apparent chaos in Figure \ref{fig:Simul_2_zoom_3}.\\
\newline
Let us also emphasize that despite we clearly observed isolated $\omega$-limit sets for special values of $r$, we did not detect any periodic pattern for the restitution coefficients close to these special values. It is possible that no stable periodic pattern exists in this region, and that all the trajectories end up attracted by a pattern that is not periodic, but stable. It is also possible that several unstable periodic patterns exist for the same value of $r$, and that each trajectory wanders around a self-similar, periodic trajectory for a while, until it leaves it for another periodic trajectory, over and over again.\\
In any case, the change of shape of the $\omega$-limit sets depending on $r$ is an intriguing phenomenon. The fact that all the $\omega$-limit sets seem to be finite unions of intervals, that might be extremely small, or extremely large, suggest that the trajectories of the spherical reduction mapping might be quasi-periodic. The $\omega$-limit sets might indeed be invariant tori, and the fact that their ``diameters'' changes with $r$ might be only a consequence of the projection of the trajectories on the $(r,\varphi)$ plane.

\section{Conclusions and outlook}
\noindent
In this work, we have examined a new pattern in which inelastic collapse can occur, which is periodic and asymmetric (in the sense that more collisions occur for the leftmost particles than the rightmost). We found an explicit initial datum that leads to collapse when $r<r_{crit.,\mathfrak{ababcb}}=5-2\sqrt{6}$. However, we also find that this explicit initial datum is unstable, in the sense that starting at some other initial datum in a small neighbourhood around the explicit, self-similar datum will generate some other sequence of collisions (that is not the periodic pattern $\mathfrak{ababcb}$) eventually. It would be interesting to study if any such periodic and asymmetric pattern is always unstable, thereby narrowing the possible (stable) patterns of collisions that can lead to inelastic collapse.\newline
In general, there are still several open problems, even for four particles: what is the upper threshold of $r$ for which inelastic collapse can occur (stable or not)? Additionally, it is not known if the periodic patterns of the form $(\mathfrak{ab})^n(\mathfrak{cb})^n$ are the only patterns that are stable in some region of $r$, and what exactly their intervals of stability are (the upper threshold of stability is conjectured to correspond to the condition $q_2(0)=0$ for the dominant eigenvector of the corresponding matrix, see \cite[Section 7]{CDKK999}). Numerically, their intervals of stability seem to be disjoint and accumulate from above on $r_{crit,3\text{ balls}}=7-4\sqrt{3}$, the threshold for three-ball inelastic collapse, but this has also not been proven yet.\newline
Moreover, to the best of our knowledge, there exists no result concerning non-periodic patterns, despite the fact that those kind of patterns might be the only observable ones outside the stability windows. To describe such mechanisms of collapse, one would need to develop new mathematical tools.
\newline
Additionally, there are more open questions that remain, concerning systems with five or more particles: what patterns lead to inelastic collapse and for which ranges of the restitution coefficient? Which of them are stable? Does there exist a similar family of patterns that are stable in some intervals, that accumulate on some threshold for collapse with one less particle? \newline
\newline
We have also introduced a new representation for four inelastic particles evolving on a one-dimensional line by reducing it to a billiard (with an unusual reflection law) on a portion of the sphere $\mathbb{S}^2$. This reduction can also equivalently be seen as a dynamical system $\Phi: \{1,2,3\}\times \mathbb{S}^2\rightarrow\{1,2,3\}\times \mathbb{S}^2$. By reducing the dimension of the system, we were able to perform efficient simulations.\newline
We have proven that collisions of a given trajectory of the original system with four particles occur in the same order for the associated trajectory, computed with the spherical reduction mapping, so that it suffices to understand the possible patterns that the spherical reduction generates. Another interesting feature is that the spherical reduction always gives an infinite sequence of ``collisions'', for arbitrary values of the restitution coefficient $r$, even if the real particles of the original inelastic four-hard-sphere system separate after some time. Nonetheless, we showed that the order of the generated sequences of ``collisions'' are not fundamentally different from the real system, and our numerical investigations on the spherical reduction mapping agree with the results in the literature.\\
It remains a lot to understand the complicated dynamics of the spherical reduction mapping. We hope that further theoretical studies of this dynamical system might help understand better the inelastic collapse phenomenon. \newline 
Moreover, we identified numerically some regimes in which our reduction has some apparent quasi-periodic behaviour that seems to be unknown so far. It is possible that invariant tori, in which the trajectories of the system lie, exist in the phase space. These regimes lie above the currently known threshold below which inelastic collapse can occur in a stable way ($r=3-2\sqrt{2}$), but below the currently known threshold where inelastic collapse can occur (stable or not; $r\simeq 0.19166$). It remains to study the precise connection between these structures and the dynamics of the original system of particles. \newline
Finally, it would be interesting to study if there exists a similar spherical reduction for systems with five or more particles.\newline

\textbf{Acknowledgements.} The authors are grateful to J. J. L. Vel\'azquez for many stimulating discussions concerning the topic of the present article. The authors are also grateful to M. Breden, for his comments concerning the interpretations of the numerical simulations.\\
\newline
The first author gratefully acknowledges the financial support, on the one hand when he was affiliated to the University of Bonn until August 2024, of the Hausdorff Research Institute for Mathematics through the collaborative research center The mathematics of emerging effects (CRC 1060, Project-ID 211504053), and the Deutsche Forschungsgemeinschaft (DFG, German Research Foundation), and on the other hand since September 2024, of the project PRIN 2022 (Research Projects of National Relevance) - Project code 202277WX43, based at the University of L'Aquila.\\
The second author gratefully acknowledges the support by the Graduiertenkolleg 2339 IntComSin of the Deutsche Forschungsgemeinschaft (DFG, German
Research Foundation, Project-ID 321821685).

E-mail address: \texttt{theophile.dolmaire@univaq.it}, \texttt{eleni.huebner-rosenau@ur.de}.

\end{document}